\documentclass[12pt,reqno]{amsart}
\usepackage{amssymb,amsmath,amsthm,enumerate,bbm}
\usepackage[a4paper]{geometry}
\usepackage{color}

\DeclareMathOperator{\Ran}{Ran}
\DeclareMathOperator{\Dom}{Dom}
\DeclareMathOperator{\Ker}{Ker}

\DeclareMathOperator{\rank}{rank}

\renewcommand\Im{\hbox{{\rm Im}}\,}

\newcommand{\abs}[1]{\lvert#1\rvert}
\newcommand{\Abs}[1]{\left\lvert#1\right\rvert}
\newcommand{\norm}[1]{\lVert#1\rVert}

\newcommand{\jap}[1]{\langle#1\rangle}

\newcommand{\bbT}{{\mathbb T}}
\newcommand{\bbR}{{\mathbb R}}
\newcommand{\bbC}{{\mathbb C}}

\newcommand{\bbD}{{\mathbb D}}
\newcommand{\bbP}{\mathbb {P}}

\newcommand{\calH}{{\mathcal H}}
\newcommand{\calP}{\mathcal{P}}
\newcommand{\calR}{\mathcal{R}}
\newcommand{\calS}{\mathcal{S}}
\newcommand{\calA}{\mathcal{A}}

\numberwithin{equation}{section}

\theoremstyle{plain}
\newtheorem{theorem}{\bf Theorem}[section]
\newtheorem*{theorem*}{Theorem 1.1$'$}
\newtheorem{lemma}[theorem]{\bf Lemma}

\theoremstyle{remark}
\newtheorem*{remark*}{\bf Remark}
\newtheorem{remark}[theorem]{\bf Remark}

\newcommand{\wh}{\widehat}

\newcommand{\eps}{\varepsilon}
\newcommand{\1}{\mathbbm{1}}

\newcommand{\la}{\langle}
\newcommand{\ra}{\rangle}
\newcommand{\e}{\varepsilon}

\newcommand{\oH}{\overset{\circ}{H}}
\newcommand{\ou}{\overset{\circ}{u}}

\begin{document}

\title[Cubic Szeg\H{o} equation on the line]{The cubic Szeg\H{o} equation on the real line: explicit formula and well-posedness on the Hardy class}

\author{Patrick G\'erard}
\address{Universit\'e Paris-Saclay, Laboratoire de Math\'ematiques d'Orsay, CNRS, UMR 8628, France}
\email{patrick.gerard@universite-paris-saclay.fr}

\author{Alexander Pushnitski}
\address{Department of Mathematics, King's College London, Strand, London, WC2R~2LS, U.K.}
\email{alexander.pushnitski@kcl.ac.uk}

\subjclass[2000]{Primary 47B35, secondary 37K20}

\keywords{Cubic Szeg\H{o} equation, evolution flow, Hankel operators}

\date{July 2023}

\begin{abstract}
We establish an explicit formula for the solution of the the cubic Szeg\H{o} equation on the real line. Using this formula, we prove that the evolution flow of this equation can be continuously extended to the whole Hardy class $H^2$ on the real line. 
\end{abstract}

\maketitle

\setcounter{tocdepth}{1}
\tableofcontents
\setcounter{tocdepth}{3}

\section{Introduction and main result}\label{sec.a}

\subsection{The  cubic Szeg\H{o} equation on $\bbR$}\label{sec.a1}

For $p\geq1$, let  $H^p(\bbR)$  be the standard Hardy class in the upper half-plane, defined as the set of all functions $f=f(z)$, analytic in $\Im z>0$ and having the finite norm
\[
\norm{f}_{H^p(\bbR)}^p=\sup_{y>0}\int_{-\infty}^\infty \abs{f(x+iy)}^p dx.
\]
We will routinely identify functions $f\in H^p(\bbR)$ with their boundary values on the real line, and with this identification one has $H^p(\bbR)\subset L^p(\bbR)$ as a closed subspace. If $p=2$, this is a Hilbert space with the inner product (linear in the first argument and anti-linear in the second argument) denoted by $\jap{\cdot,\cdot}$; this is identical to the restriction of the $L^2$ inner product. 
Let $\bbP$ be the orthogonal projection onto $H^2(\bbR)$ in $L^2(\bbR)$.

In \cite{OP1,OP2} Pocovnicu, by analogy with the unit circle case \cite{GG1,GG2} introduced and studied the \emph{cubic Szeg\H{o} equation} 
\begin{equation}
i\frac{\partial u}{\partial t}
=
\bbP(\abs{u}^2u), \quad u=u(x,t), \quad x,t\in \bbR.
\label{a1}
\end{equation}
Here,  for every $t\in \bbR$, the function $u(\cdot,t)$ belongs to a suitable Sobolev class of functions in $H^2(\bbR)$. 
More precisely, in \cite{OP1} it was proven to be well-posed on the intersection of $H^2(\bbR)$ with the Sobolev class $W^{s,2}(\bbR)$ with the smoothness exponent $s\geq1/2$.
The cubic Szeg\H{o} equation is easily seen to be the Hamiltonian equation with respect to the symplectic form
$$
\omega(u,v)=\Im \jap{u,v}_{H^2}
$$
on $H^2(\bbR)$ and the Hamiltonian 
$$
E(u)=\frac14\norm{u}_{H^4(\bbR)}^4.
$$
Furthermore, following the strategy of the unit circle case \cite{GG1,GG2}, it was proven
in \cite{OP1,OP2} that this equation is completely integrable and possesses a Lax pair. The Lax pair involves \emph{Hankel operators} on $H^2(\bbR)$ and will be discussed in Section~\ref{sec.b}. 

\subsection{Main results}
The purpose of this paper is two-fold: to establish an \emph{explicit formula for the solution} of the cubic Szeg\H{o} equation \eqref{a1} and to extend its flow to the whole Hardy space $H^2(\bbR)$. 
The formula we establish (see Theorem~\ref{Szegoline}) is heavily based on the Lax pair structure for the cubic Szeg\H{o} equation. It involves the Hankel operator with the symbol $u(\cdot,0)$. 
It will be convenient to state this formula  in Section~\ref{sec.b}  after all the ingredients have been introduced.

Here we describe the second main result. 
Let $\Phi(t)$, $t>0$, be the (non-linear) flow map of \eqref{a1}, defined on $H^2(\bbR)\cap W^{s,2}(\bbR)$ with $s\geq1/2$, i.e. if $u=u(x,t)$ is a solution to \eqref{a1}, then 
$$
\Phi(t)u(\cdot,0)=u(\cdot,t). 
$$
\begin{theorem}\label{thm.a1}
The flow map $\Phi(t)$ can be extended to $H^2(\bbR)$ such the following properties are satisfied:
\begin{enumerate}[\rm (i)]
\item
$\norm{\Phi(t)u}_{H^2}=\norm{u}_{H^2}$ for all $u\in H^2(\bbR)$;
\item
if $\norm{u_n-u}_{H^2}\to0$, then $\norm{\Phi(t)u_n-\Phi(t) u}_{H^2}\to0$ as $n\to\infty$;
\item
the map $t\mapsto\Phi(t)u$ is continuous in $H^2(\bbR)$ for all $u\in H^2(\bbR)$;
\item
if $u\in H^4(\bbR)$, then $\Phi(t)u\in H^4(\bbR)$ and $E(\Phi(t)u)=E(u)$. 
\end{enumerate}
\end{theorem}

An important ingredient of the proof is the fact that our explicit formula for the flow map $\Phi(t)$ makes sense on the whole Hardy class $H^2(\bbR)$.   
This forces us to consider unbounded Hankel operators with symbols in $H^2(\bbR)$ in Section~\ref{sec.c}. It is convenient to postpone the description of the structure of the rest of the paper until Section~\ref{sec.b7}.

\subsection{The unit circle case}
The Szeg\H{o} equation \eqref{a1} was originally introduced in  \cite{GG1,GG2}, where it was considered for functions $u$ defined on the unit circle $\bbT$. 
In this case, $\bbP$ is the orthogonal projection in $L^2(\bbT)$ onto the Hardy class $H^2(\bbT)$. 
In fact, \cite{GG1,GG2} provided the blueprint for the study of the Lax pair structure of \cite{OP1,OP2}. 
More precisely, in the unit circle case the Szeg\H{o} equation is completely integrable and possesses a Lax pair, 
which involves a Hankel operator in the Hardy class of the unit disk. There is an explicit formula \cite{GG} for the flow map of the cubic Szeg\H{o} equation on the unit circle, which involves Hankel operators and is based on the Lax pair structure. Using this formula, in \cite{GP} we have extended the flow map of the cubic Szeg\H{o} equation on the unit circle to the whole Hardy class $H^2(\bbT)$. 
We were also able to prove that this is sharp in the sense that the flow map does not extend to any Sobolev class with a negative smoothness exponent. We do not know whether Theorem~\ref{thm.a1} is sharp in the same sense.


\subsection{Methods used in the proof}
The proof of the explicit formula is based on the Lax pair structure of the cubic Szeg\H{o} equation and on the commutation properties of Hankel and Toeplitz operators with the infinitesimal generator of the Lax-Beurling semigroup in $H^2(\bbR)$, see Section~\ref{sec.b}. 
It has a different nature compared to the analogous formula \cite{GG} on the unit circle and is not reducible to it. 

This difference reflects the fact that despite the similarities between the cubic Szeg\H{o} equation on the unit circle and on the real line, the dynamics have some very different features. For example, on the real line, for a dense subset of rational initial data solutions,  have Sobolev norms tending to infinity as $t\to\infty$ \cite{GP-JEMS}, while on the circle, solutions with rational initial data are quasi-periodic \cite{GG}. 

Similarly, Theorem~\ref{thm.a1} is not reducible to the corresponding result of \cite{GP} (even though we are using some elements of \cite{GP}). 
The main new ingredients include spectral theory of dissipative operators and semigroup theory, combined with the Wold decomposition for contractions.

\subsection{Relation to other results.} Even for integrable PDEs, the existence of an explicit formula for the general solution is a rare phenomenon. However, such a formula  has been recently proved in the case of the Benjamin--Ono equation \cite{G}, both on the unit circle and on the line. Another example is the Calogero--Moser derivative nonlinear Schr\"odinger equation \cite{GL} and its periodic version, the Calogero--Sutherland equation \cite{B}, where the explicit formula was used to prove global wellposedness in critical regularity. In the latter case, the explicit formula was used to display growth of Sobolev norms for multisoliton solutions.
Indeed one expects that such explicit formulae can help understand the long time dynamics of integrable PDEs, which is still a widely open problem for many of these equations, in particular for the Szeg\H{o} equation on the line.\\
The sharp regularity wellposedness for other integrable Hamiltonian PDEs has recently stimulated the development of new methods based on the Lax pair structure, see in particular \cite{NRT} for the Davey--Stewartson equation, \cite{KT,HKV,KL}
for cubic nonlinear Schr\"odinger equations, \cite{HKNV} for the derivative nonlinear Schr\"odinger equation, and \cite{GKT,KLV} for the Benjamin--Ono equation.\\
Finally, let us mention that the cubic Szeg\H{o} equation on the line recently appeared \cite{C} in the description, through modified scattering, of the long time dynamics of the weakly dispersive nonlinear cubic Schr\"odinger equation on the plane $\bbR_x\times \bbR_y$,
$$i\partial_tu+\partial_x^2u-|D_y|u=\lambda |u|^2u\ .$$

\subsection{Notation}
The upper half-plane is denoted by $\bbC_+=\{z\in\bbC: \Im z>0\}$, and the unit disk by $\bbD=\{\zeta\in\bbC: \abs{\zeta}<1\}$. 
The Hardy class in the unit disk $H^2(\bbT)$ is defined as the set of all functions $f$ analytic in the unit disk and having the finite norm
\[
\norm{f}_{H^2(\bbD)}^2:=\sup_{r<1}\int_{-\pi}^\pi\abs{f(re^{i\theta})}^2\frac{d\theta}{2\pi}. 
\]
We denote by $S$  the standard shift operator in $H^2(\bbT)$, acting by $Sf(\zeta)=\zeta f(\zeta)$, and let $S^*$ is its adjoint. 

We denote by $\calR$ the class of rational functions in $H^2(\bbR)$; explicitly, elements of $\calR$ are functions of the form $p(z)/q(z)$, where $p$ and $q$ are polynomials with $\deg p<\deg q$ and all zeros of $q$ are in the lower half-plane $\Im z<0$.
For any $a\in\bbC_+$, we denote by $v_a$ the reproducing kernel of $H^2(\bbR)$:
$$
v_a(x)=\frac1{2\pi i}\frac1{\overline{a}-x}, 
$$
so that $f(a)=\jap{f,v_a}$ for any $f\in H^2(\bbR)$. Observe that 
$$
\norm{v_a}^2=\jap{v_a,v_a}=v_a(a)=\frac1{4\pi \Im a}.
$$
The Fourier transform on the real line is defined by 
$$
\wh f(\xi)=\int_{-\infty}^\infty f(x)e^{-ix\xi}dx\, .
$$
The identity operator is denoted by $I$. For elements $\psi,\varphi$ in a Hilbert space, we denote by $\jap{\cdot,\psi}\varphi$ the rank one operator in this space acting as $f\mapsto\jap{f,\psi}\varphi$. 

We freely use the operator theoretic background collected in Appendix~A.

\subsection{Acknowledgements}
The authors are grateful to Sandrine Grellier for useful discussions.

\section{Explicit formula for the flow map $\Phi(t)$}\label{sec.b}

The main result of this section is Theorem~\ref{Szegoline}, which gives an explicit formula for the flow map $\Phi(t)$.
We start with some preliminaries, introducing the objects that enter this explicit formula.

\subsection{The Lax-Beurling semigroup}
In the space $H^2(\bbR)$, we consider the Lax-Beurling semigroup $S_\tau$, 
$$
(S_\tau f)(z)=e^{i\tau z}f(z), \quad z\in\bbC_+, \quad \tau>0,
$$
and its adjoint $S_\tau^*$. We denote by $A$ the infinitesimal generator of $S_\tau$; it has the domain 
$$
\Dom A=\{f\in H^2(\bbR): \tfrac{d}{d\xi}\widehat f(\xi)\in L^2(0,\infty),\  \widehat f(0)=0\}\, , 
$$
and acts by $Af(z)=zf(z)$ for $f\in \Dom A$. 
The adjoint $A^*$ has the domain 
$$
\Dom A^*=\{f\in H^2(\bbR): \tfrac{d}{d\xi}\widehat f(\xi)\in L^2(0,\infty)\}\, .
$$
\begin{remark*}
One can understand the difference between $\Dom A$ and $\Dom A^*$ by looking at the Fourier images  $\wh S_\tau$ and $\wh S_\tau^*$, which act as right and left shifts on $L^2(0,\infty)$. Their generators correspond to differentiation. Now observe that if $f$ is smooth on $[0,\infty)$, extended by zero to $(-\infty,0)$,  and $\tau\to0$,  then $\frac1{\tau}(f(x+\tau)-f(x))$  converges to $f'(x)$ for all $x>0$ , while 
$\frac1{\tau}(f(x-\tau)-f(x))$ has a $\delta$-type singularity near $x=0$ unless $f(0)=0$.

We also illustrate the difference between $\Dom A$ and $\Dom A^*$ by looking at the set of rational functions $\calR\subset H^2(\bbR)$. We have $\calR\subset\Dom A^*$, while $\calR\not\subset\Dom A$; indeed, $f\in\calR$ belongs to $\Dom A$ if and only if $f(x)=O(1/x^2)$ as $x\to\infty$. 
\end{remark*}

For any point $z_0\in\bbC_+$, the operators $A-\overline{z_0}$ and $A^*-z_0$ are boundedly invertible; we will need formulas for the inverses:
\begin{equation}
(A-\overline{z_0})^{-1}f(z)=(z-\overline{z_0})^{-1}f(z), 
\quad
(A^*-z_0)^{-1}f(z)=\frac{f(z)-f(z_0)}{z-z_0}.
\label{b1a}
\end{equation}

\subsection{Hankel and Toeplitz operators}
For a \emph{symbol} $u\in H^2(\bbR)\cap H^\infty(\bbR)$ we define the anti-linear Hankel operator $H_u$ on $H^2(\bbR)$ by
$$
H_u f=\bbP(u\overline{f}), \quad f\in H^2(\bbR);
$$
the conjugate of $f$ is always taken on the real line.
It is clear that the boundedness of $u$ ensures that the operator $H_u$ is bounded on $H^2(\bbR)$. Observe that $H_u$ satisfies the symmetry relation
$$
\jap{H_uf,g}_{H^2}=\jap{\bbP(u\overline{f}),g}_{H^2}=\jap{u\overline{f},g}_{L^2}=\jap{u\overline{g},f}_{L^2}=\jap{H_ug,f}_{H^2}.
$$
We will often use the square $H_u^2$;  while $H_u$ is anti-linear, its square is linear. Moreover, by the above symmetry relation we have
$$
\jap{H_u^2f,f}=\jap{H_u(H_uf),f}=\jap{H_uf,H_uf}\geq0,
$$
and so $H_u^2$ is a self-adjoint positive semi-definite operator.

For a symbol $\varphi\in L^\infty(\bbR)$, the Toeplitz operator $T_\varphi$ on $H^2(\bbR)$ is defined by 
$$
T_\varphi f=\bbP(\varphi f), \quad f\in H^2(\bbR).
$$
Again, the boundedness of $\varphi$ ensures that $T_\varphi$ is bounded on $H^2(\bbR)$. 

\subsection{Lax pair for the cubic Szeg\H{o} equation}

We use the following result of \cite{OP1}:

\begin{theorem}\cite{OP1}
\label{thm.b1}
Let $u_t(x)=u(x,t)$ be a solution to the cubic Szeg\H{o} equation with $u_0\in W^{\frac12,2}(\bbR)\cap H^2(\bbR)$. Then the Hankel operator $H_{u_t}$ satisfies the equation
\begin{equation}
\frac{d}{dt}H_{u_t}=[B_{u_t}, H_{u_t}]\ ,
\label{b3}
\end{equation}
where $B_{u_t}$ is the skew-adjoint operator
$$
B_u=-iT_{|u|^2}+\frac i2H_u^2\ .
$$
\end{theorem}

\subsection{The functional $I_+$}
For $f\in\Dom A^*$, one may define
$$I_+(f):=\wh f(0^+)\ .$$
For $\eps>0$, let $\chi_\eps\in H^2(\bbR)$ be the \emph{approximate identity}
$$\chi_\eps (x):=\frac{1}{1-i\eps x}\ .$$
Then by Plancherel's theorem we have 
$$
\jap{f,\chi_\eps}=\frac1{\eps}\int_0^\infty \wh f(\xi)e^{-\xi/\eps}d\xi,
$$
and therefore 
\begin{equation}
I_+(f)=\lim_{\eps\to0}\jap{f,\chi_\eps}
\label{b4}
\end{equation}
for any $f\in\Dom A^*$. 

Here is another way of expressing $I_+$ on the domain of $A^*$. Fix $z\in\bbC_+$ and write an arbitrary element in $\Dom A^*$ as $(A^*-z)^{-1}f$ for $f\in H^2(\bbR)$; then by \eqref{b1a},
\begin{align}
I_+((A^*-z)^{-1}f)
&=x
\lim_{\eps\to0}\jap{(A^*-z)^{-1}f,\chi_\eps}
=
\lim_{\eps\to0}\jap{f,(A-\overline{z})^{-1}\chi_\eps}
\notag
\\
&=
2\pi i\jap{f,v_z}=2\pi if(z).
\label{b4a}
\end{align}

\subsection{Explicit formula for the flow map}
We fix $t>0$ once and for all. 
For a symbol $u\in H^2(\bbR)\cap H^\infty(\bbR)$ we introduce the notation 
\begin{equation}
L_u(t)=
\frac{1}{2\pi}\int_0^t\jap{\cdot,\mathrm{e}^{-is H_{u}^2}u}\mathrm{e}^{-is H_{u}^2}u\, ds\, .
\label{d1}
\end{equation}
Here the integrand is a self-adjoint rank one operator, and therefore $L_u(t)$ is a self-adjoint trace class operator. 
\begin{theorem}\label{Szegoline}
Let $\Phi(t)$ be the flow map defined on $H^2(\bbR)\cap W^{s,2}(\bbR)$ with $s>1/2$. 
Then, for every $z\in \bbC_+ $ and every $u\in H^2(\bbR)\cap W^{s,2}(\bbR)$, 
\begin{equation}
\boxed{
(\Phi(t)u_0)(z)=\frac{1}{2\pi i}I_+\left [\left (A^*+L_{u_0}(t)- zI \right )^{-1}\mathrm{e}^{-itH_{u_0}^2}u_0\right ] .}
\label{d1a}
\end{equation}
\end{theorem}
Some remarks are in order concerning the meaning of formula \eqref{d1a}. Since $L_{u_0}(t)$ is bounded and self-adjoint, the operator $A+L_{u_0}(t)$, defined on $\Dom A$, is maximal dissipative (see e.g. \cite[Section 3.1]{Davies}). It follows that $-(A^*+L_{u_0}(t))$ is also maximal dissipative and therefore the operator $A^*+L_{u_0}(t)- zI$ in \eqref{d1a} has a bounded inverse. Furthermore,
$$
\Ran (A^*+L_{u_0}(t)- zI)^{-1}=\Dom(A^*+L_{u_0}(t)-zI)=\Dom A^*
$$ 
and so the functional $I_+$ in \eqref{d1a} is well-defined. 

We also observe that in the context of this theorem, we require $s>1/2$ because in the course of the proof we need the boundedness of $u$. 

\subsection{The structure of the rest of the paper}\label{sec.b7}

In the rest of this section, we prove Theorem~\ref{Szegoline}; its proof is a calculation based on the Lax pair formulation \eqref{b3} of the cubic Szeg\H{o} equation.

Our next step is to set up the map $\Phi(t)$ on the whole Hardy class $H^2(\bbR)$. Since formula \eqref{d1a} for $\Phi(t)$ involves a Hankel operator, we must first extend the notion of Hankel operators $H_u$ to arbitrary symbols $u\in H^2(\bbR)$. This is done in Section~\ref{sec.c}. After this, in Section~\ref{sec.d} we define the map $\Phi(t)$ on $H^2(\bbR)$ by the same formula  \eqref{d1a}xx. Our key result is Theorem~\ref{thm.d1} which states that thus defined, $\Phi(t)$ preserves the $H^2(\bbR)$-norm. Assuming this result, it is very easy to complete the proof of Theorem~\ref{thm.a1}; we do this in Section~\ref{sec.d}. In Sections~\ref{sec.e}--\ref{sec.g} we prove Theorem~\ref{thm.d1}; the proof involves passing to the Cayley transform of $A+L_{u_0}(t)$ and using some abstract operator theoretic theorem of \cite{GP}.

\subsection{Two auxiliary identities}
The following lemma can be found in \cite{Sun,GL}, but for completeness we give an alternative proof below.
\begin{lemma}\label{bracket0}
Let $b\in L^2(\bbR)\cap L^\infty(\bbR)$. Then for every $f\in \Dom A^*$, we have $T_bf\in \Dom A^*$ and
$$[A^*,T_b]f= \frac{i}{2\pi}I_+(f)\bbP b\ .$$
\end{lemma}
\begin{proof}
\emph{Step 1:}
Assume $b\in H^2(\bbR)\cap H^\infty(\bbR)$. Let us prove that for every $f\in H^2(\bbR)$ and for every $z\in\bbC_+$, 
$$
[T_b,(A^*-i)^{-1}]f(z)=-f(i)(A^*-i)^{-1}b(z).
$$
In this case for every $z\in\bbC_+$, we have $T_bf(z)=b(z)f(z)$ and using the explicit formula \eqref{b1a} for the resolvent of $A^*$, we find
\begin{align*}
T_b(A^*-i)^{-1}f(z)&-(A^*-i)^{-1}T_bf(z)
=
b(z)\frac{f(z)-f(i)}{z-i}-\frac{b(z)f(z)-b(i)f(i)}{z-i}
\\
&=
-\frac{b(z)-b(i)}{z-i}f(i)=-f(i)(A^*-i)^{-1}b(z),
\end{align*}
as claimed. 

\emph{Step 2:} Assume $b\in H^\infty(\bbR)$. Let us prove that 
$$
[T_{\overline{b}},(A^*-i)^{-1}]=0.
$$
Indeed, for any $f\in H^2(\bbR)$ and for any $z\in\bbC_+$, we find
\begin{align*}
\jap{[T_{\overline{b}},(A^*-i)^{-1}]f,v_z}
&=
\jap{T_{\overline{b}}(A^*-i)^{-1}f,v_z}
-
\jap{(A^*-i)^{-1}T_{\overline{b}}f,v_z}
\\
&=
\jap{(A^*-i)^{-1}f,bv_z}
-
\jap{T_{\overline{b}}f,(A+i)^{-1}v_z}
\\
&=
\jap{f,(A+i)^{-1}bv_z}
-
\jap{f,b(A+i)^{-1}v_z}=0,
\end{align*}
as claimed. 

\emph{Step 3:}
Let $b\in L^2(\bbR)\cap L^\infty(\bbR)$. Let us prove that for any $f\in H^2(\bbR)$, 
\begin{equation}
[T_b,(A^*-i)^{-1}]f=-f(i)(A^*-i)^{-1}\bbP b.
\label{b5}
\end{equation}
For $\eps>0$, let $b_\eps$ be the convolution of $b$ with the Poisson kernel, 
$$
b_\eps(x)=\frac{\eps}{\pi}\int_{-\infty}^\infty \frac{b(y)}{(x-y)^2+\eps^2}dy. 
$$
Then for each $\eps>0$, we have $b_\eps=b_{\eps,+}+b_{\eps,-}$, where $b_{\eps,+}=\bbP b_\eps\in H^2(\bbR)\cap H^\infty(\bbR)$ and $\overline{b_{\eps,-}}\in H^\infty(\bbR)$. Combining the two previous steps, we find 
\begin{equation}
[T_{b_\eps},(A^*-i)^{-1}]f=-f(i)(A^*-i)^{-1}\bbP b_\eps.
\label{b6}
\end{equation}
As $\eps\to0$, we have $\norm{\bbP b_\eps-\bbP b}_{H^2}\to0$ and $T_{b_\eps}\to T_b$ in the weak operator topology. Passing to the limit as $\eps\to0_+$ in \eqref{b6}, we arrive at \eqref{b5}.

\emph{Step 4:}
Let us apply \eqref{b5} to the vector $f=(A^*-i)g$ for some $g\in \Dom A^*$. We find
$$
T_bg-(A^*-i)^{-1}T_b(A^*-i)g=-f(i)(A^*-i)^{-1}\bbP b,
$$
and by \eqref{b4a} 
$$
f(i)=\frac1{2\pi i}I_+(g).
$$
It follows that $T_bg\in\Dom A^*$ and 
$$
A^*T_bg-T_bA^*g=
(A^*-i)T_bg-T_b(A^*-i)g=-\frac1{2\pi i}I_+(g)\bbP b,
$$
as required. 
\end{proof}
Below $B_u=-iT_{|u|^2}+\frac i2H_u^2$ is the operator from Theorem~\ref{thm.b1}. 

\begin{lemma}\label{bracket}
Let $u\in H^2(\bbR)\cap H^\infty(\bbR)$; then for every $f\in\Dom A^*$, we have
$B_uf\in\Dom A^*$ and $H_u^2f\in\Dom A^*$ and 
$$[{A^*},B_u]f=-\frac{i}{2}[{A^*},H_u^2]f+\frac{1}{2\pi}\jap{f, u}u\, .$$
\end{lemma}
\begin{proof}
Let us first check the identity 
$$
H_u^2=T_{|u|^2}-T_uT_{\overline u}.
$$
Evaluating the quadratic forms, we find
\begin{align*}
\jap{H_u^2f,f}
&=
\norm{\bbP(u\overline{f})}^2
=
\norm{u\overline{f}}^2-\norm{\bbP(\overline{u}f)}^2
=
\norm{\overline{u}f}^2-\norm{\bbP(\overline{u}f)}^2
\\
&=
\jap{\abs{u}^2f,f}-\norm{T_{\overline{u}}f}^2
=
\jap{T_{\abs{u}^2}f,f}-\jap{T_u T_{\overline{u}}f,f},
\end{align*}
as claimed.

By Lemma~\ref{bracket0}, it follows that both $H_u^2$ and $B_u$ map $\Dom A^*$ into itself. 
Furthermore
$$B_u+\frac{i}{2}H_u^2=-iT_{|u|^2}+iH_u^2=-iT_uT_{\overline u}$$
and so for $f\in \Dom A^*$, 
$$
[{A^*},B_u]f=-\frac{i}{2}[{A^*},H_u^2]f -i[{A^*},T_uT_{\overline u}]f\, .
$$
Now using Lemma~\ref{bracket0} again, we find
$$
[A^*,T_uT_{\overline u}]f
=
[A^*,T_u]T_{\overline u}f+T_u[A^*,T_{\overline u}]f
=\frac{i}{2\pi}I_+(T_{\overline u}f)u,
$$
and 
$$
I_+(T_{\overline u}f)
=
\lim_{\eps\to0_+}\jap{T_{\overline u}f,\chi_\eps}
=
\lim_{\eps\to0_+}\jap{f,u\chi_\eps}
=
\jap{f,u}\, .
$$
Putting this together gives the required identity.
\end{proof}

\subsection{Proof of Theorem~\protect\ref{Szegoline}}
Proceeding as in \cite{GG}, we introduce the family $U(t)$ of unitary operators on $H^2(\bbR)$ defined by the linear initial value problem
$$U'(t)=B_{{u_t}}U(t)\ ,\ U(0)=I\ .$$
Using the Lax pair equation \eqref{b3}, we observe that 
$$
\frac{d}{dt}U(t)^*H_{{u_t}}U(t)=-U(t)^*B_{u_t}H_{u_t}U(t)+U(t)^*[B_{u_t},H_{u_t}]U(t)+U(t)^*H_{u_t}B_{u_t}U(t)=0
$$
and therefore
$$
U(t)^*H_{{u_t}}U(t)=H_{u_0}.
$$
Next, we have 
\begin{equation}
\frac{d}{dt}U(t)^*\chi_\eps=-U(t)^*B_{u_t}\chi_\eps=iU(t)^*(T_{\abs{u_t}^2}\chi_\eps-\frac12H_{u_t}^2\chi_\eps).
\label{b8}
\end{equation}
For the r.h.s. here we have
$$
\lim_{\eps\to0_+}T_{\abs{u_t}^2}\chi_\eps
=
\bbP(\abs{u_t}^2)=H_{u_t}u_t,
$$
and 
$$
\lim_{\eps\to0_+}H_{u_t}^2\chi_\eps=\lim_{\eps\to0_+}H_{u_t}(H_{u_t}\chi_\eps)=H_{u_t}u_t,
$$
and therefore the r.h.s. of \eqref{b8} converges to 
$$
\frac{i}{2}U(t)^*H_{u_t}u_t
=
\frac{i}{2}\lim_{\eps\to0_+}U(t)^*H_{u_t}^2\chi_\eps
=
\frac{i}{2}\lim_{\eps\to0_+}H_{u_0}^2U(t)^*\chi_\eps.
$$
By integrating in time, we infer
$$
U(t)^*\chi_{\eps} -\mathrm{e}^{i\frac t2 H_{u_0}^2}\chi_{\eps} \rightarrow 0
$$
in $H^2(\bbR)$.
Applying $H_{u_0}$ to both sides here, we find
\begin{align*}
H_{u_0}U(t)^*\chi_{\eps}
&=
U(t)^*H_{u_t}\chi_\eps\to U(t)^*u_t\, ,
\\
H_{u_0}\mathrm{e}^{i\frac t2 H_{u_0}^2}\chi_{\eps}
&=
\mathrm{e}^{-i\frac t2 H_{u_0}^2}H_{u_0} \chi_{\eps}
\to
\mathrm{e}^{-i\frac t2 H_{u_0}^2}u_0\, ,
\end{align*}
and therefore
\begin{equation}
U(t)^*u_t=\mathrm{e}^{-i\frac t2 H_{u_0}^2}u_0\, .
\label{b10}
\end{equation}

Then using Lemma~\ref{bracket} and \eqref{b10},  we calculate
\begin{align*}
\frac{d}{dt}U(t)^*A^*U(t)&=U(t)^*[{A^*},B_{{u_t}}]U(t)
\\
&=-\frac i2U(t)^*[{A^*},H_{{u_t}}^2]U(t) +\frac{1}{2\pi}U(t)^*\bigl(\jap{\cdot, {u_t}}{u_t}\bigr)U(t)
\\
&=-\frac{i}{2}[U(t)^*A^*U(t),H_{u_0}^2]+\frac{1}{2\pi}\jap{\cdot, \mathrm{e}^{-i\frac{t}{2}H_{u_0}^2}u_0} \mathrm{e}^{-i\frac{t}{2}H_{u_0}^2}u_0 \ .
\end{align*}
Integrating this ODE and recalling the definition \eqref{d1} of $L_{u_0}(t)$ as an integral of rank one operators, we get
$$
U(t)^*A^*U(t)=\mathrm{e}^{i\frac t2 H_{u_0}^2}\left ({A^*}+L_{u_0}(t)\right )\mathrm{e}^{-i\frac t2 H_{u_0}^2}\ .
$$
Plugging this information into formula \eqref{b4a} with $f=u_t$, we infer
\begin{align*}
u_t(z)&=\frac{1}{2\pi i}\lim_{\e \to 0} \jap{\left({A^*}-zI\right)^{-1}{u_t}, \chi_\e} 
\\
&=\frac{1}{2\pi i}\lim_{\e \to 0}
\jap{ \left(U(t)^*A^*U(t)-zI\right)^{-1} U(t)^*{u_t}, U(t)^*\chi_\e}
\\
&= \frac{1}{2\pi i}\lim_{\e \to 0} 
\left \la \left (\mathrm{e}^{i\frac t2 H_{u_0}^2}\left ({A^*}+L_{u_0}(t)\right )\mathrm{e}^{-i\frac t2 H_{u_0}^2}-zI\right )^{-1}\mathrm{e}^{-i\frac{t}{2}H_{u_0}^2}u_0, \mathrm{e}^{i\frac{t}{2}H_{u_0}^2}\chi _\e \right \ra 
\\
&= \frac{1}{2\pi i}\lim_{\e \to 0} 
\left \la \left ({A^*}+L_{u_0}(t)-zI\right )^{-1}\mathrm{e}^{-itH_{u_0}^2}u_0, \chi _\e \right \ra 
\\
&= \frac{1}{2\pi i} I_+\left [\left({A^*}+L_{u_0}(t)-zI\right )^{-1}\mathrm{e}^{-itH_{u_0}^2}u_0\right ]\ .
\end{align*}
The proof of Theorem~\ref{Szegoline} is complete. \qed

\section{Unbounded Hankel operators in $H^2(\bbR)$}\label{sec.c}

\subsection{The set-up and the key result}
We need to extend the notion of Hankel operators $H_u$ in $H^2(\bbR)$ to the case of symbols $u\in H^2(\bbR)$. 
In this case the operators $H_u$ may become unbounded and so we need to discuss the issues of domains, closability etc.

Let $u\in H^2(\bbR)$. We define the Hankel operator $H_uf=\bbP(u \overline{f})$ on the domain 
\begin{equation}
\Dom H_u=\{f\in H^2(\bbR): \bbP(u \overline{f})\in H^2(\bbR)\}.
\label{c1}
\end{equation}
Here condition $\bbP(u \overline{f})\in H^2(\bbR)$ should be understood as follows: for the product $h:=u\overline{f}\in L^1(\bbR)$, the Fourier transform $\widehat h$, restricted onto the positive semi-axis, belongs to $L^2(0,\infty)$. 

We will also need a dense in $H^2(\bbR)$ subset of ``nice'' functions that are in the domain of every Hankel operator; we use the set of all rational functions $\calR\subset H^2(\bbR)$ for this purpose.

\begin{theorem}\label{thm.c1}
Let $u\in H^2(\bbR)$, and let $H_uf=\bbP(u \overline{f})$ be the anti-linear Hankel operator with the domain \eqref{c1}. Then: 
\begin{enumerate}[\rm (i)]
\item
$H_u$ is closed, i.e. $\Dom H_u$ is closed with respect to the graph norm of $H_u$, 
$$
\norm{f}_{H_u}:=\bigl(\norm{f}_{H^2}^2+\norm{H_uf}_{H^2}^2\bigr)^{1/2}\, ;
$$

\item
the set of rational functions $\calR$ is dense in $\Dom H_u$ with respect to the graph norm;
\item
for any $f,g\in\Dom H_u$, we have 
$$
\jap{H_uf,g}=\jap{H_ug,f};
$$
\item
suppose for some $f,h\in H^2(\bbR)$  we have 
$$
\jap{H_ug,f}=\jap{h,g}, \quad \forall g\in\Dom H_u.
$$
Then $f\in\Dom H_u$ and $H_uf=h$. 
\end{enumerate}
\end{theorem}
The proof is given at the end of this section.
\subsection{Hankel operators in $H^2(\bbT)$}
In \cite{GP}, we proved analogous results for Hankel operators acting on the Hardy space of the unit disk $H^2(\bbT)$. 
Our proof of Theorem~\ref{thm.c1} consists of mapping the results of \cite{GP} onto our situation using the conformal map from the unit disc to the upper half-plane. We start by introducing relevant notation for the unit circle case. 

Let $H^2(\bbT)\subset L^2(\bbT)$ be the standard Hardy class of the unit disk, and let $\overset{\circ}{\bbP}:L^2(\bbT)\to H^2(\bbT)$ be the orthogonal projection onto this space (we will supply some notation related to the unit circle case with a circle above the letter). For a symbol $\ou\in H^2(\bbT)$, we denote by $\oH_{\ou}$ the (possibly unbounded) anti-linear Hankel operator in $H^2(\bbT)$, defined by 
\begin{equation}
\oH_{\ou}f=\overset{\circ}{\bbP}(\ou\overline{f})\, ,
\label{c3}
\end{equation}
with the domain
\begin{equation}
\Dom \oH_{\ou}=\{f\in H^2(\bbT): \overset{\circ}{\bbP}(\ou\overline{f})\in H^2(\bbT)\}.
\label{c4}
\end{equation}
Here condition $\overset{\circ}{\bbP}(\ou\overline{f})\in H^2(\bbT)$ should be understood as follows: for the function $h=\ou\overline{f}\in L^1(\bbT)$, the Fourier coefficients $\wh h_n$ with $n\geq0$ are square-summable. Let $\calP$ be the set of all polynomials in $H^2(\bbT)$; it plays the role of the dense set of ``nice'' functions in $H^2(\bbT)$, which are in the domain of every Hankel operator. 

We quote the relevant results of \cite{GP} in this notation.
\begin{theorem}\label{thm.c1a}
Let $\ou\in H^2(\bbT)$, and let $\oH_{\ou}$ be the Hankel operator defined by \eqref{c3} with the domain \eqref{c4}. 
Then:
\begin{enumerate}[\rm (i)]
\item
$\oH_{\ou}$ is closed, i.e. $\Dom \oH_{\ou}$ is closed with respect to the graph norm of $\oH_{\ou}$;
\item
the set of polynomials $\calP$ is dense in $\Dom \oH_{\ou}$ with respect to the graph norm;
\item
for any $f,g\in\Dom \oH_{\ou}$, we have 
$$
\jap{\oH_{\ou}f,g}=\jap{\oH_{\ou}g,f};
$$
\item
suppose for some $f,h\in H^2(\bbT)$  we have 
$$
\jap{\oH_{\ou}g,f}=\jap{h,g}, \quad \forall g\in\Dom \oH_{\ou}.
$$
Then $f\in\Dom \oH_{\ou}$ and $\oH_{\ou}f=h$. 
\end{enumerate}
\end{theorem}
\subsection{Conformal map}
Let $\varphi$ be the standard conformal map from the upper half-plane onto the unit disk, $\varphi(z)=\frac{z-i}{z+i}$, and let $W$ be the corresponding unitary map from $H^2(\bbT)$ onto $H^2(\bbR)$, 
\begin{equation}
(Wf)(z)=\frac1{\sqrt{\pi}(z+i)}f(\varphi(z)), \quad z\in\bbC_+.
\label{c6}
\end{equation}
We denote by $S$ be the standard shift operator in $H^2(\bbT)$.

\begin{lemma}\label{lma.c2}
Let $u\in H^2(\bbR)$ and let $\ou=S^*(u\circ\varphi^{-1})$. Then
for any $f\in H^2(\bbT)$, the inclusions $f\in \Dom \oH_{\ou}$ and $Wf\in\Dom H_u$ are equivalent.
Furthermore, for any $f\in \Dom \oH_{\ou}$, the identity
\begin{equation}
W^*H_uWf=\oH_{\ou}f
\label{c7}
\end{equation}
holds true. 
\end{lemma}
Before embarking on the proof, we note that by a change of variable, 
\begin{equation}
\int_{-\pi}^\pi \abs{u\circ\varphi^{-1}(e^{i\theta})}^2d\theta
=
2\int_{-\infty}^\infty \frac{\abs{u(x)}^2}{1+x^2}dx,
\label{c9}
\end{equation}
we see that $u\in H^2(\bbR)$ implies that $u\circ\varphi^{-1}\in H^2(\bbT)$ and therefore $\ou\in H^2(\bbT)$. 
\begin{proof}
Let $f\in H^2(\bbT)$ and $h\in\calP$; by a change of variable we have
\begin{align}
\int_{-\infty}^\infty &u(x)\overline{Wf(x)}\overline{Wh(x)}dx
=
\int_{-\infty}^\infty u(x)\overline{f(\varphi(x))}\overline{h(\varphi(x))}\frac{dx}{\pi(x-i)^2}
\notag
\\
&=
\int_{-\infty}^\infty \frac{x+i}{x-i}u(x)\overline{f(\varphi(x))}\overline{h(\varphi(x))}\frac{dx}{\pi(x^2+1)}
\notag
\\
&=
\int_{-\pi}^\pi e^{-i\theta} u(\varphi^{-1}(e^{i\theta}))\overline{f(e^{i\theta})}\overline{h(e^{i\theta})}\frac{d\theta}{2\pi}
\notag
\\
&=
\int_{-\pi}^\pi \ou(e^{i\theta})\overline{f(e^{i\theta})}\overline{h(e^{i\theta})}\frac{d\theta}{2\pi}\, .
\label{c8}
\end{align}
Observe that $\calP$ is dense in $H^2(\bbT)$ and as a consequence, $W\calP$ is dense in $H^2(\bbR)$. Now the inclusion $f\in \Dom \oH_{\ou}$ is equivalent to the estimate 
$$
\Abs{\int_{-\pi}^\pi \ou(e^{i\theta})\overline{f(e^{i\theta})}\overline{h(e^{i\theta})}\frac{d\theta}{2\pi}}\leq\norm{h}_{H^2(\bbT)},\quad \forall\  h\in\calP
$$
while the inclusion $Wf\in\Dom H_u$ is equivalent to the estimate
$$
\Abs{\int_{-\infty}^\infty u(x)\overline{Wf(x)}\overline{Wh(x)}dx}
\leq 
C\norm{Wh}_{H^2(\bbR)}, \quad \forall\  h\in \calP.
$$
By \eqref{c8}, the left hand sides of these two estimates coincide, and by the isometricity of $W$ the right hand sides coincide as well. Thus, the inclusions $f\in \Dom \oH_{\ou}$ and $Wf\in\Dom H_u$ are equivalent to one another. For $f\in \Dom \oH_{\ou}$, identity \eqref{c8} yields \eqref{c7}.
\end{proof}

\subsection{Proof of Theorem~\protect\ref{thm.c1}}
Items (i), (iii), (iv) of Theorem~\ref{thm.c1} follow from the corresponding items of 
Theorem~\ref{thm.c1a}. By Theorem~\ref{thm.c1a}, the set $W\calP\subset\calR$ is dense in $\Dom H_u$ with respect to the graph norm of $H_u$; it follows that $\calR$ is also dense in the same sense.
\qed

\subsection{Strong resolvent convergence}
For $u\in H^2(\bbR)$, we define the linear operator $H_u^2$ on the domain
$$
\Dom H_u^2=\{f\in H^2(\bbR): f\in\Dom H_u \text{ and }H_uf\in\Dom H_u\}.
$$
This operator is closed; furthermore, it is self-adjoint and positive semi-definite. It can be alternatively described as the self-adjoint operator corresponding to the quadratic form $\norm{H_u f}^2$.

\begin{theorem}\label{thm.c2}
Let $\{u_n\}_{n=1}^\infty$ be a sequence of elements of $H^2(\bbR)$ with $\norm{u_n-u}_{H^2}\to0$ as $n\to\infty$. Then 
$$
H_{u_n}^2\to H_u^2
$$
in the strong resolvent sense. 
\end{theorem}
\begin{proof}
The proof follows from Lemma~\ref{lma.c2} and from the corresponding statement \cite[Theorem 2.2]{GP} about Hankel operators in $H^2(\bbT)$: 
\emph{if $\{\ou_n\}_{n=1}^\infty$ is a sequence of elements of $H^2(\bbT)$ with $\norm{\ou_n-\ou}_{H^2}\to0$ as $n\to\infty$, then
$$
(\oH_{\ou_n})^2\to (\oH_{\ou})^2
$$
in the strong resolvent sense.}
We only need to comment that by \eqref{c9}, convergence 
$\norm{u_n-u}_{H^2}\to0$ in $H^2(\bbR)$ implies the convergence
$\norm{\ou_n-\ou}_{H^2}\to0$ in $H^2(\bbT)$. 
\end{proof}

\section{The explicit formula for initial data from $H^2(\bbR)$}\label{sec.d}

\subsection{The flow preserves the norm in $H^2(\bbR)$}
We keep $t>0$ fixed and continue using the notation $L_u(t)$ for the integral \eqref{d1} with any symbol $u\in H^2(\bbR)$; here $H_u$ is the closed operator as defined in the previous section. 
By the remarks following Theorem~\ref{Szegoline}, the r.h.s. of \eqref{d1a} still makes sense in this context. Where no confusion arises, we write $L_u$ in place of $L_u(t)$. 
The main ingredient of the proof of Theorem~\ref{thm.a1} is 

\begin{theorem}\label{thm.d1}
For any $u\in H^2(\bbR)$ and $t>0$, let $\Phi(t)u$ be the holomorphic function in the 
upper half-plane, defined by 
\begin{equation}
(\Phi(t)u)(z)
=
\frac{1}{2\pi i}I_+\left [\left (A^*+L_{u}(t)- z \right )^{-1}\mathrm{e}^{-itH_{u}^2}u\right ] , \quad \Im z>0.
\label{d2}
\end{equation}
Then $\Phi(t)u\in H^2(\bbR)$ and 
$$
\norm{\Phi(t)u}_{H^2}=\norm{u}_{H^2}. 
$$
\end{theorem}

In the rest of this Section, assuming Theorem~\ref{thm.d1}, we prove Theorem~\ref{thm.a1}. 
Theorem~\ref{thm.d1} is proved in Sections~\ref{sec.e}--\ref{sec.g}.

We emphasize that in Theorem~\ref{thm.d1}, formula \eqref{d2} is regarded as a \emph{definition} of a non-linear map $\Phi(t)$ on $H^2(\bbR)$. The fact that $\Phi(t)$, thus defined, coincides with the flow map of the cubic Szeg\H{o} equation is established later. 

\subsection{Rewriting $\Phi(t)$ without $I_+$}
Below we recast the definition \eqref{d2} in an alternative way, avoiding the use of the functional $I_+$. 
Recall that $v_a$ is the reproducing kernel of $H^2(\bbR)$ at the point $a$ in the upper half-plane. In particular, below we use $v_i$ which corresponds to the choice $a=i$. 

\begin{lemma}\label{lma.d2}
For any $u\in H^2(\bbR)$ and any $\Im z>0$, the r.h.s. of \eqref{d2} can be written as
\begin{equation}
(\Phi(t)u)(z)
=
\frac1{\sqrt{4\pi}}\jap{(A^*+L_u-i)(A^*+L_u-z)^{-1}e^{-itH_u^2}u,q}, 
\label{d3}
\end{equation}
where
$$
q:=\sqrt{4\pi}(I-(A+L_u+i)^{-1}L_u)v_i.
$$
\end{lemma}
We note that the reason for the ``strange'' normalisation of the vector $q$ is that $\norm{q}=1$ (as we shall see later). This vector plays an important role in what follows. 
\begin{proof} 
By formula \eqref{b1a} for the resolvent of $A$, we have 
$$
\frac1{2\pi i}\lim_{\eps\to0_+}(A+i)^{-1}\chi_\eps=-v_i\, .
$$
Next, we recall that $\Dom A^*=\Dom (A^*+L_u)$ (this will be used in the calculation below). Using \eqref{b4}, we find
\begin{align*}
(\Phi(t)u)(z)
&=
\frac{1}{2\pi i}I_+\left [\left (A^*+L_{u}- z \right )^{-1}\mathrm{e}^{-itH_{u}^2}u\right ] 
\\
&=
\frac{1}{2\pi i}\lim_{\eps\to0_+} \jap{(A^*+L_{u}- z)^{-1}\mathrm{e}^{-itH_{u}^2}u,\chi_\eps}
\\
&=
\frac{1}{2\pi i}\lim_{\eps\to0_+} \jap{(A^*+L_{u}- z)^{-1}\mathrm{e}^{-itH_{u}^2}u,(A+i)(A+i)^{-1}\chi_\eps}
\\
&=
\frac{1}{2\pi i}\lim_{\eps\to0_+} \jap{(A^*-i)(A^*+L_{u}- z)^{-1}\mathrm{e}^{-itH_{u}^2}u,(A+i)^{-1}\chi_\eps}
\\
&=
\jap{(A^*-i)(A^*+L_{u}- z)^{-1}\mathrm{e}^{-itH_{u}^2}u,v_i}\, .
\end{align*}
Finally, substituting 
\begin{align*}
(A^*-i)(A^*+L_{u}- z)^{-1}
&=
(A^*+L_u-i-L_u)(A^*+L_{u}- z)^{-1}
\\
&=
(I-L_u(A^*+L_{u}-i)^{-1})(A^*+L_u-i)(A^*+L_{u}- z)^{-1}
\end{align*}
in the previous identity, we obtain the required result. 
\end{proof}

\subsection{A convergence argument}
We isolate a convergence argument as a lemma. 
\begin{lemma}\label{lma.d6}
Let $u_n\in H^2(\bbR)$ be a convergent sequence, $\norm{u_n-u}_{H^2}\to0$ as $n\to\infty$. Then for all $t\in\bbR$ we have the $H^2(\bbR)$-convergence 
$$
\norm{e^{-itH_{u_n}^2}u_n-e^{-itH_{u}^2}u}_{H^2}\to0
$$
and the operator norm convergence
$$
\norm{L_{u_n}-L_{u}}\to0
$$
as $n\to\infty$. 
\end{lemma}
In fact, the above convergence also holds true in the trace norm, but we will not need this fact. 
\begin{proof}
By Theorem~\ref{thm.c2} we have the strong resolvent convergence 
$$
H_{u_n}^2\to H_u^2. 
$$ 
This implies \cite[Theorem VIII.20]{RS1} that for all bounded continuous functions $f$, we have $f(H_{u_n}^2)\to f(H_u^2)$ in the strong operator topology. In particular, 
$$
e^{-itH_{u_n}^2}\to e^{-itH_{u}^2}
$$
for all $t\in\bbR$ strongly as $n\to\infty$.
It follows that 
$$
\norm{e^{-itH_{u_n}^2}u_n-e^{-itH_{u}^2}u}_{H^2}\to0
$$
for all $t\in\bbR$ as $n\to\infty$ and therefore we have the operator norm convergence of rank one operators
$$
\norm{\jap{\cdot,e^{-itH_{u_n}^2}u_n}e^{-itH_{u_n}^2}u_n-
\jap{\cdot,e^{-itH_{u}^2}u}e^{-itH_{u}^2}u}\to0, \quad n\to\infty.
$$
Also, the operator norms of $\jap{\cdot,e^{-itH_{u_n}^2}u_n}e^{-itH_{u_n}^2}u_n$ are uniformly bounded. Integrating over $t$ and using the operator norm valued dominated convergence theorem, we find that 
$$
\norm{L_{u_n}-L_u}\to0, \quad n\to\infty.
$$
The proof is complete.
\end{proof}

\subsection{Proof of Theorem~\protect\ref{thm.a1}}
Assuming Theorem~\ref{thm.d1}, let us prove Theorem~\ref{thm.a1}.
We extend the map $\Phi(t)$ to the whole of $H^2(\bbR)$ as indicated in Theorem~\ref{thm.d1}. 
Clearly, (i) follows directly from Theorem~\ref{thm.d1}. 

Let us prove (ii). Let $\norm{u_n-u}_{H^2}\to0$. 
Fix $z\in\bbC_+$ and let us pass to the limit in \eqref{d3}. We have
$$
(A^*+L_u-i)(A^*+L_u-z)^{-1}=I+(z-i)(A^*+L_u-z)^{-1}
$$
and by the resolvent identity, 
\begin{equation}
(A^*+L_{u_n}-z)^{-1}-(A^*+L_u-z)^{-1}
=
(A^*+L_{u_n}-z)^{-1}(L_u-L_{u_n})(A^*+L_u-z)^{-1}. 
\label{d6}
\end{equation}
Applying the resolvent norm estimate \eqref{b0a} of a maximal dissipative operator to the r.h.s. of \eqref{d6} (where $-A^*-L_{u}$ and $-A^*-L_{u_n}$ are maximal dissipative) and using Lemma~\ref{lma.d6},  we find that 
$$
\norm{(A^*+L_{u_n}-i)(A^*+L_{u_n}-z)^{-1}e^{-itH_{u_n}^2}u_n-(A^*+L_u-i)(A^*+L_u-z)^{-1}e^{-itH_u^2}u}\to0
$$
as $n\to\infty$. Finally, let 
$$
q=\sqrt{4\pi}(I-(A+L_u+i)^{-1}L_u)v_i, \quad q_n=\sqrt{4\pi}(I-(A+L_{u_n}+i)^{-1}L_{u_n})v_i.
$$
Again by the resolvent identity and by Lemma~\ref{lma.d6}, we find that $\norm{q_n-q}_{H^2}\to0$ as $n\to\infty$. 
We conclude that we can pass to the limit in \eqref{d3}, i.e. 
$$
\abs{(\Phi(t)u_n)(z)-(\Phi(t)u)(z)}\to0, \quad n\to\infty,
$$
for all $\Im z>0$, and so $\Phi(t)u_n\to\Phi(t)u$ weakly in $H^2(\bbR)$ as $n\to\infty$. 
On the other hand, by Theorem~\ref{thm.d1} we have 
$$
\norm{\Phi(t)u_n}_{H^2}=\norm{u_n}_{H^2}\to\norm{u}_{H^2}=\norm{\Phi(t)u}_{H^2}
$$
and so we conclude that $\norm{\Phi(t)u_n-\Phi(t)u}_{H^2}\to0$ as $n\to\infty$.

Let us prove (iii). Let $t_n\to t$ as $n\to\infty$. 
Following exactly the same argument as in the previous step of the proof (or simply rescaling by writing $u_n=t_n u$), we find that $\Phi(t_n)u\to\Phi(t)u$ weakly as $n\to\infty$ and on the other hand, by Theorem~\ref{thm.d1} 
$$
\norm{\Phi(t_n)u}_{H^2}=\norm{u}_{H^2}=\norm{\Phi(t)u}_{H^2}
$$
and so again we conclude that $\norm{\Phi(t_n)u-\Phi(t)u}_{H^2}\to0$ as $n\to\infty$.

Let us prove (iv). We split the proof into four steps. 

\emph{Step 1:}
Let $u\in H^2(\bbR)$; consider the Fourier transform of $\abs{u}^2$, evaluated on the real axis. We have
$$
\widehat{\abs{u}^2}(-\xi)=\overline{\widehat{\abs{u}^2}(\xi)}, \quad \xi\in\bbR,
$$
and therefore, using Plancherel's theorem, we find that $u\in H^4$ iff
$$
\norm{\bbP(\abs{u}^2)}_{H^2}^2=\int_0^\infty \Abs{\widehat{\abs{u}^2}(\xi)}^2d\xi<\infty.
$$
Furthermore, this argument also shows that if $u\in H^4(\bbR)$, then 
$$
\norm{u}_{H^4}^4=2\norm{\bbP(\abs{u}^2)}_{H^2}^2. 
$$
We conclude that $u\in H^4(\bbR)$ iff $u\in\Dom H_u$ and in this case 
$$
\norm{u}_{H^4}^4=2\norm{H_u u}_{H^2}^2.
$$

\emph{Step 2:}
For $x>0$ and $u\in H^2(\bbR)$, denote
$$
J(x,u):=\jap{(I+xH_u^2)^{-1}u,u}. 
$$
The function $x\mapsto J(x,u)$ is $C^\infty$ smooth on $(0,\infty)$, with 
$$
\partial_x J(x,u)=-\jap{(I+xH_u^2)^{-1}H_u^2(I+xH_u^2)^{-1}u,u}
=-\norm{H_u(I+xH_u^2)^{-1}u}^2.
$$
Observe that $u\in\Dom H_u$ iff
$$
\sup_{x>0}\norm{H_u(I+xH_u^2)^{-1}u}<\infty,
$$
i.e. iff the derivative $\partial_x J(x,u)$ remains bounded as $x\to0_+$, and in this case
$$
\partial_x J(0_+,u)=-\norm{H_uu}^2.
$$
Combining this with the previous step, we see that $u\in H^4(\bbR)$ iff $\partial_x J(x,u)$ remains bounded as $x\to0_+$, and in this case
$$
\partial_x J(0_+,u)=-\tfrac12\norm{u}_{H^4}^4.
$$

\emph{Step 3:}
Let us check that 
\begin{equation}
J(x,\Phi(t)u)=J(x,u), \quad\forall u\in\calR, \quad \forall x>0.
\label{d7}
\end{equation}
Denote $u_t(z)=(\Phi(t)u)(z)$; since $u$ is rational, we can freely use the cubic Szeg\H{o} equation \eqref{a1} and the Lax pair formulation \eqref{b3}, and so it suffices to check that 
$$
\partial_t J(x,u_t)=0.
$$
The cubic Szeg\H{o} equation \eqref{a1} can be written as
$$
\partial_t u_t=-i\bbP(\abs{u_t}^2u_t)=-iT_{\abs{u_t}^2}u_t=B_{u_t}u_t-\tfrac{i}{2}H_{u_t}^2u_t.
$$
On the other hand, from the Lax pair equation \eqref{b3} by induction in $m$ we find 
$$
\partial_t H_{u_t}^m=[B_{u_t},H_{u_t}^m]
$$
and therefore
$$
\partial_t (I+xH_{u_t}^2)^{-1}=[B_{u_t},(I+xH_{u_t}^2)^{-1}]\, .
$$
Putting this together, using the fact that $B_{u}^*=-B_u$ and that $H_{u_t}^2$ commutes with $(I+xH_{u_t}^2)^{-1}$, we find
\begin{align*}
\partial_t J(x,u_t)=&\jap{[B_{u_t},(I+xH_{u_t}^2)^{-1}]u_t,u_t}
\\
&+\jap{(I+xH_{u_t}^2)^{-1}(B_{u_t}-\tfrac{i}{2}H_{u_t}^2)u_t,u_t}
\\
&+\jap{(I+xH_{u_t}^2)^{-1}u_t,(B_{u_t}-\tfrac{i}{2}H_{u_t}^2)u_t}
=0,
\end{align*}
as claimed. 

\emph{Step 4:}
By the already proven part (ii) of Theorem~\ref{thm.a1}, we can extend \eqref{d7} from the dense set of $u\in\calR$ to all $u\in H^2(\bbR)$. Now we can put it all together. In steps 1 and 2, we have characterised the inclusion $u\in H^4(\bbR)$ in terms of the function $J(x,u)$. In step 3, we proved that this function is invariant under the map $\Phi(t)$. This concludes the proof of Theorem~\ref{thm.a1}(iv). \qed

\section{In the land of contractions}\label{sec.e}
In the rest of the paper, our task is to prove Theorem~\ref{thm.d1}. 
Formula \eqref{d2} for $\Phi(t)$ is stated in terms of the (adjoint of) dissipative operator $A+L_u$. 
Our strategy is to rewrite this formula in terms of a contraction $\Sigma_u$ which is a Cayley transform of $A+L_u$ and then to use (i) a commutation relation between $\Sigma_u$ and a certain auxiliary Hankel-like operator $\calH_u$ and (ii) a general operator theoretic lemma about contractions from our previous paper \cite{GP}. 

Figuratively speaking, at this point in the paper we are crossing the bridge from the land of dissipative operators to the land of contractions.

\subsection{Contractions $\Sigma_0$ and $\Sigma_u$}
Here we define the Cayley transforms of $A$ and $A+L_u$ and discuss their basic properties. 
Let $u\in H^2(\bbR)$ and let $L_u$ be the trace-class self-adjoint operator \eqref{d1}.
We define
\begin{equation}
\Sigma_0:=(A-i)(A+i)^{-1}, \quad \Sigma_u:=(A+L_u-i)(A+L_u+i)^{-1}.
\label{e8aa}
\end{equation}
We will regard $\Sigma_u$ as a ``perturbation'' of $\Sigma_0$. 
\begin{lemma}\label{lma.d3}
The operator $\Sigma_0$ is a completely non-unitary contraction, unitarily equivalent to the shift operator $S$ in $H^2(\bbT)$ and satisfying
\begin{equation}
\Sigma_0^*\Sigma_0=I, \quad \Sigma_0\Sigma_0^*=I-4\pi \jap{\cdot,v_i}v_i.
\label{e8a}
\end{equation}
\end{lemma}
\begin{proof}
By \eqref{b1a}, the operator $\Sigma_0$ is the operator of multiplication by the function $\frac{x-i}{x+i}$ in $H^2(\bbR)$. 
Let $W:H^2(\bbT)\to H^2(\bbR)$ be the unitary operator \eqref{c6}. 
Changing the variable, we find that 
$$
W^*\Sigma_0W=S,
$$
i.e. the usual shift operator in $H^2(\bbT)$. In particular, $\Sigma_0$ is a completely non-unitary contraction with defect indices $(0,1)$. For $S$ we have
$$
SS^*=I-\jap{\cdot,\1}\1,
$$
where $\1\in H^2(\bbT)$ is the function identically equal to $1$. By a direct calculation, $W\1=-2\sqrt{\pi}iv_i$, and so \eqref{e8a} follows. 
\end{proof}

\begin{lemma}\label{lma.d4}
The operator $\Sigma_u$ is a contraction, satisfying
$$
\Sigma_u^*\Sigma_u=I, \quad \Sigma_u \Sigma_u^*=I-\jap{\cdot,q}q,
$$
where $q$ is as in Lemma~\ref{lma.d2}; moreover, $\norm{q}=1$.
\end{lemma}
\begin{proof}
\emph{Step 1: representation for $\Sigma_u$.} We would like to represent $\Sigma_u$ in terms of $\Sigma_0$ and $L_u$. First we represent the operator $A+L_u$ in terms of $\Sigma_0$ and $L_u$. We recall that 
$$
\Dom(A+L_u)=\Dom A=\Ran (I-\Sigma_0)
$$
and 
$$
A=i(I+\Sigma_0)(I-\Sigma_0)^{-1} \text{ on $\Dom A$.}
$$
It follows that on $\Dom A$ we have
\begin{align*}
A+L_u\pm i
&=
i(I+\Sigma_0)(I-\Sigma_0)^{-1}+L_u\pm i
\\
&=
\bigl(i(I+\Sigma_0)+L_u(I-\Sigma_0)\pm i(I-\Sigma_0)\bigr)(I-\Sigma_0)^{-1}\, .
\end{align*}
Choosing the plus sign in the above expression, we would like to check that the operator in brackets, viz.
$$
\bigl(i(I+\Sigma_0)+L_u(I-\Sigma_0)+ i(I-\Sigma_0)\bigr)
=
2i+L_u(I-\Sigma_0),
$$
has a bounded inverse. Since $L_u$ is compact, by the Fredholm alternative it suffices to check that the kernel is trivial. Suppose $(2i+L_u(I-\Sigma_0))f=0$ for some $f\in H^2(\bbR)$; then for $g=(I-\Sigma_0)f$ we have 
$$
(A+L_u+i)g=(2i+L_u(I-\Sigma_0))(I-\Sigma_0)^{-1}g=(2i+L_u(I-\Sigma_0))f=0, 
$$
which is impossible because $A$ is maximal dissipative and therefore $\Ker(A+L_u+i)$ is trivial. 

Now we compute:
\begin{align}
\Sigma_u
&=(A+L_u-i)(A+L_u+i)^{-1}
\notag
\\
&=
(2i\Sigma_0+L_u(I-\Sigma_0))(I-\Sigma_0)^{-1}
\biggl( (2i+L_u(I-\Sigma_0))(I-\Sigma_0)^{-1}\biggr)^{-1}
\notag
\\
&=
(2i\Sigma_0+L_u(I-\Sigma_0))(2i+L_u(I-\Sigma_0))^{-1}
\notag
\\
&=
(\Sigma_0+\tfrac1{2i}L_u(I-\Sigma_0))
(I+\tfrac1{2i}L_u(I-\Sigma_0))^{-1}.
\label{d3c}
\end{align}
In the same way, starting from 
$$
A=i(I-\Sigma_0)^{-1}(I+\Sigma_0) \text{ on $\Dom A$,}
$$
we find
\begin{equation}
\Sigma_u
=
(I+\tfrac1{2i}(I-\Sigma_0)L_u)^{-1}
(\Sigma_0+\tfrac1{2i}(I-\Sigma_0)L_u)\, ,
\label{d3d}
\end{equation}
where $I+\tfrac1{2i}(I-\Sigma_0)L_u$ has a bounded inverse.

\emph{Step 2: computing $\Sigma_u^*\Sigma_u$.}
Using \eqref{d3c} and \eqref{d3d}, we write 
\begin{equation}
\Sigma_u^*\Sigma_u
=
(I-\tfrac1{2i}(I-\Sigma_0^*)L_u)^{-1}M(I+\tfrac1{2i}L_u(I-\Sigma_0))^{-1}\, ,
\label{d3a}
\end{equation}
where
\begin{align*}
M&=
(\Sigma_0^*-\tfrac1{2i}(I-\Sigma_0^*)L_u)
(\Sigma_0+\tfrac1{2i}L_u(I-\Sigma_0))
\\
&=
\Sigma_0^*\Sigma_0+\tfrac1{2i}\Sigma_0^*L_u(I-\Sigma_0)
-\tfrac1{2i}(I-\Sigma_0^*)L_u\Sigma_0
-\tfrac1{(2i)^2}(I-\Sigma_0^*)L_u^2(I-\Sigma_0).
\end{align*}
Using that $\Sigma_0^*\Sigma_0=I$, we can rewrite this as
\begin{align*}
M&=
I+\tfrac1{2i}\Sigma_0^*L_u
-\tfrac1{2i}L_u\Sigma_0
-\tfrac1{(2i)^2}(I-\Sigma_0^*)L_u^2(I-\Sigma_0)
\\
&=
I-\tfrac1{2i}(I-\Sigma_0^*)L_u
+\tfrac1{2i}L_u(I-\Sigma_0)
-\tfrac1{(2i)^2}(I-\Sigma_0^*)L_u^2(I-\Sigma_0)
\\
&=
(I-\tfrac1{2i}(I-\Sigma_0^*)L_u)
(I+\tfrac1{2i}L_u(I-\Sigma_0))
\end{align*}
and substituting this back into \eqref{d3a}, we find $\Sigma_u^*\Sigma_u=I$.

\emph{Step 3: computing $\Sigma_u\Sigma_u^*$.} 
We have
\begin{equation}
\Sigma_u\Sigma_u^*
=
(I+\tfrac1{2i}(I-\Sigma_0)L_u)^{-1}M_*(I-\tfrac1{2i}L_u(I-\Sigma_0^*))^{-1}
\label{d3b}
\end{equation}
where
\begin{align*}
M_*&=
(\Sigma_0+\tfrac1{2i}(I-\Sigma_0)L_u)
(\Sigma_0^*-\tfrac1{2i}L_u(I-\Sigma_0^*))
\\
&=
\Sigma_0\Sigma_0^*+\tfrac1{2i}(I-\Sigma_0)L_u\Sigma_0^*
-\tfrac1{2i}\Sigma_0L_u(I-\Sigma_0^*)
-\tfrac1{(2i)^2}(I-\Sigma_0)L_u^2(I-\Sigma_0^*).
\end{align*}
Using that $\Sigma_0\Sigma_0^*=I-4\pi\jap{\cdot,v_i}v_i$, we rewrite this as
\begin{align*}
M_*&=
I-4\pi\jap{\cdot,v_i}v_i+\tfrac1{2i}L_u\Sigma_0^*
-\tfrac1{2i}\Sigma_0L_u
-\tfrac1{(2i)^2}(I-\Sigma_0)L_u^2(I-\Sigma_0^*)
\\
&=
I-4\pi\jap{\cdot,v_i}v_i
-\tfrac1{2i}L_u(I-\Sigma_0^*)
+\tfrac1{2i}(I-\Sigma_0)L_u
-\tfrac1{(2i)^2}(I-\Sigma_0)L_u^2(I-\Sigma_0^*)
\\
&=
(I+\tfrac1{2i}(I-\Sigma_0)L_u)(I-\tfrac1{2i}L_u(I-\Sigma_0^*))
-4\pi\jap{\cdot,v_i}v_i.
\end{align*}
Substituting back into \eqref{d3b}, we find
$$
\Sigma_u\Sigma_u^*=I-\jap{\cdot,q}q,
$$
with 
\begin{align*}
q&=\sqrt{4\pi}(I+\tfrac1{2i}(I-\Sigma_0)L_u)^{-1}v_i
=
2i\sqrt{4\pi}(A+L_u+i)^{-1}(I-\Sigma_0)^{-1}v_i
\\
&=\sqrt{4\pi}(A+L_u+i)^{-1}(A+i)v_i
=\sqrt{4\pi}(I-(A+L_u+i)^{-1}L_u)v_i,
\end{align*}
as required. 

\emph{Step 4: computing the norm of $q$.} 
Since $\Sigma_u \Sigma_u^*\geq0$, we have $\norm{q}\leq1$. Since the spectra of $\Sigma_u^*\Sigma_u$ and $\Sigma_u \Sigma_u^*$ outside zero coincide and $q\not=0$, we find that $\norm{q}=1$.
\end{proof}

\subsection{Rewriting $\norm{\Phi(t)u}$ in terms of $\Sigma_u$}
We would like to recast the statement of Theorem~\ref{thm.d1} in terms of the contraction $\Sigma_u$. 

\begin{lemma}\label{lma.d5}
Let $\Phi(t)$ be defined as in \eqref{d2}, and $q$ be as in Lemma~\ref{lma.d2}. Then for any $u\in H^2(\bbR)$, we have
\begin{equation}
\sup_{y>0}
\int_{-\infty}^\infty \Abs{(\Phi(t)u)(x+iy)}^2dx
=
\sup_{0<r<1}\int_{-\pi}^\pi \abs{\jap{(I+re^{i\theta}\Sigma_u^*)^{-1}e^{-itH_u^2}u,q}}^2\frac{d\theta}{2\pi}.
\label{d4}
\end{equation}
\end{lemma}
\begin{proof}
We use the following well-known statement. 
Let $G=G(z)$ be a function holomorphic in the open upper half-plane and $g=g(\zeta)$ be a function holomorphic in the open unit disc such that
$$
G(z)=\tfrac1{\sqrt{\pi}}\tfrac1{z+i}g\left(\tfrac{z-i}{z+i}\right), \quad z\in\bbC_+.
$$
Then 
$$
\sup_{y>0}\int_{-\infty}^\infty \abs{G(x+iy)}^2dx
=
\sup_{0<r<1}\int_{-\pi}^\pi \abs{g(re^{i\theta})}^2\frac{d\theta}{2\pi}\, .
$$
Using the definition of $\Sigma_u$ as a Cayley transform of $A+L_u$, we compute
$$
(A^*+L_u-i)(A^*+L_u-z)^{-1}=\tfrac{2i}{z+i}(I+\tfrac{z-i}{z+i}\Sigma_u^*)^{-1}, 
$$
and therefore, by Lemma~\ref{lma.d2},
\[
(\Phi(t)u)(z)=\tfrac1{\sqrt{4\pi}}\tfrac{2i}{z+i}\jap{(I+\tfrac{z-i}{z+i}\Sigma_u^*)^{-1}e^{-itH_u^2}u,q}.
\]
Now taking 
$$
G(z)=(\Phi(t)u)(z)\quad\text{ and }\quad 
g(\zeta)=i\jap{(I+\zeta\Sigma_u^*)^{-1}e^{-itH_u^2}u,q},
$$
we obtain the required identity. 
\end{proof}

\subsection{Discussion: direction of further exposition}\label{sec.e4}
In order to prove Theorem~\ref{thm.d1}, we need to check that the r.h.s. in \eqref{d4} equals $\norm{u}^2$. 
Denote $p=e^{-itH_u^2}u$.
For the integrand in the right hand side of \eqref{d4}, we have
$$
\jap{(I+re^{i\theta}\Sigma_u^*)^{-1}p,q}
=
\sum_{n=0}^\infty (re^{i\theta})^n\jap{p,\Sigma_u^n q}
$$
and therefore 
$$
\sup_{r<1}\int_{-\pi}^\pi \abs{\jap{(I+re^{i\theta}\Sigma_u^*)^{-1}p,q}}^2\frac{d\theta}{2\pi}
=
\sum_{n=0}^\infty \abs{\jap{p,\Sigma_u^nq}}^2.
$$
Thus, our task is to show that 
$$
\sum_{n=0}^\infty \abs{\jap{p,\Sigma_u^nq}}^2
=
\norm{p}^2, \quad p=e^{-itH_u^2}u.
$$
From Lemma~\ref{lma.d4} we find that $\norm{\Sigma_u^m q}=\norm{q}=1$ for all $m$ and $\Sigma_u^*q=0$. 
It follows that for $m>n\geq0$
$$
\jap{\Sigma_u^m q,\Sigma_u^n q}=\jap{\Sigma_u^{m-n}q,q}=\jap{q,(\Sigma_u^*)^{m-n}q}=0,
$$
and therefore $\{\Sigma_u^m q\}_{m=0}^\infty$ is an orthonormal set in the Hardy space $H^2(\bbR)$.  
We need to prove that $p$ belongs to the closed linear span of this orthonormal set.

Below we use the Wold decomposition for $\Sigma_u$, see Theorem~\ref{thm.Wold} in the Appendix. 
By this theorem, it suffices to prove that $p$ belongs to the c.n.u. subspace in the Wold decomposition \eqref{e8} of $\Sigma_u$. 
Now observe that 
$$
p=e^{-itH_u^2}u=\lim_{\eps\to0_+}e^{-itH_u^2}H_u\chi_\eps.
$$
This suggests considering the operator $\calH_u:=e^{-itH_u^2}H_u$ in this context. 
In the next section, we study $\calH_u$ and prove that it satisfies a simple commutation relation with $\Sigma_u$. Using this and some abstract results from \cite{GP}, in Section~\ref{sec.g} we will prove that $\calH_u$ is \emph{diagonal} in the Wold decomposition \eqref{e8} of $\Sigma_u$. From here it will be very easy to conclude that $p$ is in the c.n.u. subspace of $\Sigma_u$, as required.

\section{A commutation relation between $\Sigma_u$ and $\calH_u$}

\subsection{The auxiliary operator $\calH_u$}
Let $u\in H^2(\bbR)$. Motivated by the previous discussion, here we introduce an auxiliary anti-linear operator whose properties mirror those of $H_u$. We set 
\begin{equation}
\calH_u=e^{-itH_u^2}H_u=H_ue^{itH_u^2}, \quad \Dom \calH_u=\Dom H_u.
\label{e1a}
\end{equation}
For the ease of further reference, we state the properties of $\calH_u$ as a lemma which mirrors Theorem~\ref{thm.c1}.
\begin{lemma}\label{lma.e6}
Let $u\in H^2(\bbR)$, and let $\calH_u$ be the anti-linear operator defined by \eqref{e1a}. 
Then: 
\begin{enumerate}[\rm (i)]
\item
$\calH_u$ is closed, i.e. $\Dom \calH_u$ is closed with respect to the graph norm of $\calH_u$;
\item
the set of rational functions $\calR$ is dense in $\Dom \calH_u$ with respect to the graph norm;
\item
for any $f,g\in\Dom \calH_u$, we have 
$$
\jap{\calH_uf,g}=\jap{\calH_ug,f};
$$
\item
suppose for some $f,h\in H^2(\bbR)$  we have 
$$
\jap{\calH_ug,f}=\jap{h,g}, \quad \forall g\in\Dom\calH_u.
$$
Then $f\in\Dom\calH_u$ and $\calH_uf=h$. 
\end{enumerate}
\end{lemma}
\begin{proof}
Since $H_u$ is closed, it is clear that $\calH_u$ is closed as well. 
Item (ii) follows directly from Theorem~\ref{thm.c1}(ii). 
Next, using Theorem~\ref{thm.c1}(iii), we find 
$$
\jap{\calH_u f,g}
=
\jap{e^{-itH_u^2}H_uf,g}
=
\jap{H_uf,e^{itH_u^2}g}
=
\jap{H_ue^{itH_u^2}g,f}
=
\jap{\calH_u g,f}
$$
for all $f,g\in\Dom \calH_u$.
Finally, under the hypothesis of (iv) we have
$$
\jap{H_ug,e^{itH_u^2}f}=\jap{e^{-itH_u^2}H_ug,f}=\jap{h,g}
$$
and so from Theorem~\ref{thm.c1}(iv) we find that $e^{itH_u^2}f\in\Dom H_u$ and 
$H_ue^{itH_u^2}f=h$, as required. 
\end{proof}

\subsection{The main result of this section}

Our task in this section is to prove the following commutation relation between $\calH_u$ and $\Sigma_u$. 
\begin{theorem}\label{thm.e1}
We have $\Sigma_u\Dom \calH_u\subset\Dom \calH_u$ and on $\Dom \calH_u$  the relation
\begin{equation}
\Sigma_u^*\calH_u=\calH_u\Sigma_u
\label{e1b}
\end{equation}
holds true.
\end{theorem}
Proving this requires some preparation. Essentially, this is a consequence of the analogous formula for $H_u$ and $\Sigma_0$, 
\begin{equation}
\Sigma_0^*H_u=H_u\Sigma_0
\label{e2}
\end{equation}
(see Remark~\ref{rmk.e2} below),
the definitions of operators $\Sigma_u$, $\calH_u$ and properties of $\Sigma_0$. We will also use an approximation argument, because it is easier to deal with bounded Hankel operators until the last moment. 

For much of this section, we deal with Hankel operators $H_u$ with rational symbols $u\in\calR$. 
By a version of Kronecker's theorem (a detailed proof can be found, e.g. in \cite[Theorem 2.1]{OP1}) in this case we have
$$
\Ran H_u\subset \calR, \quad u\in\calR.
$$
In particular, $\Ran H_u\subset\Dom A^*$ and so the operator $A^*H_u$ is well-defined and bounded on $H^2(\bbR)$. 

\subsection{Commutation relations between $A$ and $H_u$}

\begin{lemma}\label{lma.g2}
Let $u\in\calR$; on $\Dom A$ we have the relation 
$$
A^*H_u=H_uA. 
$$
\end{lemma}
\begin{proof}
It suffices to prove the identity
\begin{equation}
H_u(A+i)^{-1}f=(A^*-i)^{-1}H_uf
\label{e3}
\end{equation}
for all $f\in H^2(\bbR)$. Let us evaluate the inner products with the reproducing kernel $v_z$, $\Im z>0$:
\begin{align*}
\jap{H_u(A+i)^{-1}f,v_z}_{H^2}
&-
\jap{(A^*-i)^{-1}H_uf,v_z}_{H^2}
\\
=&
\jap{u\overline{(A+i)^{-1}f},v_z}_{L^2}
-
\jap{H_uf, (A+i)^{-1}v_z}_{H^2}
\\
=&
\jap{u,v_z(A+i)^{-1}f}_{H^2}
-
\jap{u\overline{f},(A+i)^{-1}v_z}_{L^2}
\\
=&
\jap{u,v_z(A+i)^{-1}f}_{H^2}
-
\jap{u,f(A+i)^{-1}v_z}_{H^2}.
\end{align*}
Using \eqref{b1a}, we observe that $(A+i)^{-1}$ commutes with multiplication, and therefore the right hand side here vanishes, as required.
\end{proof}
\begin{remark}\label{rmk.e2}
Expressing $\Sigma_0$ in terms of the resolvent of $A$, from \eqref{e3} one easily obtains \eqref{e2}.
However, we will not need \eqref{e2} in what follows. 
\end{remark}

\begin{lemma}
Let $u\in\calR$; we have the identity
$$
[A^*,H_u^2]H_u
=
\frac{i}{2\pi}(\jap{u,\cdot}H_uu-\jap{H_uu,\cdot}u).
$$
\end{lemma}
\begin{proof}
Let us use the identity $H_u^2=T_{\abs{u}^2}-T_uT_{\overline u}$ and Lemma~\ref{bracket0}.
We find 
\begin{align*}
[A^*,H_u^2]H_uf
&=
[A^*,T_{\abs{u}^2}]H_uf-[A^*,T_uT_{\overline u}]H_uf
\\
&=
[A^*,T_{\abs{u}^2}]H_uf-[A^*,T_u]T_{\overline u}H_uf
\\
&=
\frac{i}{2\pi}\biggl(I_+(H_uf)\bbP(\abs{u}^2)-I_+(T_{\overline u}H_uf)u\biggr).
\end{align*}
Now let us compute all entries in the r.h.s.; we have
$$
I_+(H_uf)
=\lim_{\eps\to0_+}\jap{H_uf,\chi_\eps}
=\lim_{\eps\to0_+}\jap{u\overline{f},\chi_\eps}
=\lim_{\eps\to0_+}\jap{u,f\chi_\eps}
=\jap{u,f}, 
$$
and similarly 
\begin{align*}
I_+(T_{\overline u}H_uf)
&=\lim_{\eps\to0_+}\jap{T_{\overline u}H_u f,\chi_\eps}
=\lim_{\eps\to0_+}\jap{{\overline u}H_u f,\chi_\eps}
\\
&=\lim_{\eps\to0_+}\jap{H_u f,u\chi_\eps}
=\jap{H_u f,u}
=\jap{H_uu,f}\, .
\end{align*}
Finally, $\bbP(\abs{u}^2)=H_uu$. Putting this together yields the required identity.
\end{proof}

\subsection{Commutation relations between $A+L_u$ and $\calH_u$}

\begin{lemma}
Let $u\in\calR$; we have the identity 
\begin{equation}
(A^*+L_u)\calH_u=\calH_u(A+L_u) 
\label{g1}
\end{equation}
on $\Dom A$.
\end{lemma}
\begin{proof}
\emph{Step 1:}
Let us prove the identity
\begin{equation}
(e^{itH_u^2}A^*e^{-itH_u^2}-A^*)H_u
=
\frac1{2\pi}\int_0^t e^{isH_u^2}\biggl(\jap{u,\cdot}H_uu-\jap{H_uu,\cdot}u\biggr)e^{isH_u^2}ds\, .
\label{g2}
\end{equation}
Differentiating the l.h.s. of \eqref{g2} with respect to $t$, we find
$$
\frac{d}{dt}e^{itH_u^2}A^*e^{-itH_u^2}H_u
=
-ie^{itH_u^2}[A^*,H_u^2]H_ue^{itH_u^2}. 
$$
Using the previous lemma and integrating, we arrive at \eqref{g2}.
Multiplying \eqref{g2} by $e^{-itH_u^2}$ on the left, we obtain
$$
A^*e^{-itH_u^2}H_u-e^{-itH_u^2}A^*H_u
=
\frac1{2\pi}\int_0^t e^{-i(t-s)H_u^2}\biggl(\jap{u,\cdot}H_uu-\jap{H_uu,\cdot}u\biggr)e^{isH_u^2}ds\, .
$$
\emph{Step 2:}
By the definition of $L_u$, we find 
\begin{align*}
\calH_u L_u-L_u\calH_u
=&
e^{-itH_u^2}H_uL_u
-
L_ue^{-itH_u^2}H_u
\\
=&
\frac1{2\pi}
\int_0^t e^{-i(t-s)H_u^2}(\jap{u,\cdot}H_uu)e^{isH_u^2}ds
\\
&-
\frac1{2\pi}
\int_0^1 e^{-isH_u^2}(\jap{H_uu,\cdot}u)e^{i(t-s)H_u^2}ds\, ;
\end{align*}
changing the variable in the second integral, we find
$$
\calH_u L_u-L_u\calH_u
=
\frac1{2\pi}
\int_0^t e^{-i(t-s)H_u^2}\biggl(\jap{u,\cdot}H_uu-\jap{H_uu,\cdot}u\biggr)e^{isH_u^2}ds.
$$
Comparing this with the final identity of the previous step, we arrive at 
$$
A^*e^{-itH_u^2}H_u-e^{-itH_u^2}A^*H_u
=
\calH_u L_u-L_u\calH_u\, .
$$
Finally, on $\Dom A$ we can use Lemma~\ref{lma.g2}, and the left hand side becomes 
$$
A^*\calH_u-\calH_uA.
$$
This yields the required identity \eqref{g1}.
\end{proof}

\subsection{Approximation arguments}

\begin{lemma}
For any $u\in H^2(\bbR)$ and $f\in\calR$, we have the identity
\begin{equation}
(A^*+L_u-i)^{-1}\calH_uf=\calH_u(A+L_u+i)^{-1}f\, .
\label{g3}
\end{equation}
\end{lemma}
\begin{proof}
\emph{Step 1:}
Take a sequence $u_n\in\calR$ such that $\norm{u_n-u}_{H^2}\to0$ as $n\to\infty$. 
Since $f\in\calR\subset H^2\cap H^\infty$, we have $\norm{(u_n-u)\overline{f}}_{L^2}\to0$ and therefore $H_{u_n}f\to H_uf$ in $H^2$. 
Furthermore, by Theorem~\ref{thm.c2}, we have the strong resolvent convergence $H_{u_n}^2\to H_{u}^2$. It follows that $\calH_{u_n}f\to\calH_u f$. Furthermore, by Lemma~\ref{lma.d6} we have the operator norm convergence
$$
\norm{L_{u_n}-L_u}\to0, \quad n\to\infty.
$$
Using the resolvent identity, we find that 
$$
\norm{(A+L_{u_n}+i)^{-1}-(A+L_u+i)^{-1}}\to0, \quad n\to\infty.
$$
\emph{Step 2:}
We write the commutation relation \eqref{g1} with $u_n$ in place of $u$ and take inverses; this yields
\begin{equation}
(A^*+L_{u_n}-i)^{-1}\calH_{u_n}f=\calH_{u_n}(A+L_{u_n}+i)^{-1}f
\label{e9}
\end{equation}
for each $n$. Our aim is to pass to the limit here. 
 Consider both sides of this identity. By Step 1 of the proof,  we have
\begin{align*}
\norm{(A+L_{u_n}+i)^{-1}f-(A+L_{u}+i)^{-1}f}&\to0, 
\\
\norm{(A^*+L_{u_n}-i)^{-1}\calH_{u_n}f-(A^*+L_{u}-i)^{-1}\calH_{u}f}&\to0, 
\end{align*}
as $n\to\infty$.
Take $g\in\calR$ and consider the inner product of \eqref{e9} with $g$:
$$
\jap{(A^*+L_{u_n}-i)^{-1}\calH_{u_n}f,g}=\jap{\calH_{u_n}(A+L_{u_n}+i)^{-1}f,g}.
$$
By the symmetry of $\calH_{u_n}$ (Lemma~\ref{lma.e6}(iii)), the r.h.s. here rewrites as 
$$
\jap{\calH_{u_n}(A+L_{u_n}+i)^{-1}f,g}
=
\jap{\calH_{u_n}g,(A+L_{u_n}+i)^{-1}f}.
$$
Putting this together and taking the limit as $n\to\infty$, we find
\begin{align*}
\jap{(A^*+L_u-i)^{-1}\calH_uf,g}
&=
\lim_{n\to\infty}
\jap{(A^*+L_{u_n}-i)^{-1}\calH_{u_n}f,g}
\\
&=
\lim_{n\to\infty}
\jap{\calH_{u_n}g,(A+L_{u_n}+i)^{-1}f}
\\
&=
\jap{\calH_{u}g,(A+L_{u}+i)^{-1}f}, 
\end{align*}
resulting in 
\begin{equation}
\jap{(A^*+L_u-i)^{-1}\calH_uf,g}
=
\jap{\calH_{u}g,(A+L_{u}+i)^{-1}f}.
\label{e10}
\end{equation}
By Lemma~\ref{lma.e6}(i) and (ii), the operator $\calH_u$ is closed and rational functions are dense in $\Dom\calH_u$ with respect to the graph norm. It follows that \eqref{e10} can be extended from $g\in\calR$ to all $g\in\Dom\calH_u$.  
From here, by Lemma~\ref{lma.e6}(iv) we find that $(A+L_{u}+i)^{-1}f\in\Dom\calH_u$ and \eqref{g3} holds. 
\end{proof}

\subsection{Proof of Theorem~\ref{thm.e1}}
Let $f\in\Dom\calH_u$; we need to prove that $\Sigma_u f\in\Dom \calH_u$ and $\calH_u \Sigma_u f=\Sigma_u^*\calH_u f$. Essentially, we will prove this by taking adjoints. 

Let $g\in\calR$; since 
$$
\Sigma_u=I-2i(A+L_u+i)^{-1}, 
$$
from \eqref{g3} we obtain
$$
\Sigma_u^*\calH_u g=\calH_u \Sigma_ug.
$$
Now using the symmetry of $\calH_u$, we find
\begin{align*}
\jap{\calH_ug,\Sigma_uf}
&=
\jap{\Sigma_u^*\calH_u g,f}
=
\jap{\calH_u\Sigma_ug,f}
\\
&=
\jap{\calH_u f,\Sigma_u g}
=
\jap{\Sigma_u^*\calH_uf, g}.
\end{align*}
Since $\calH_u$ is closed (see Lemma~\ref{lma.e6}(i)), the last relation can be extended to all $g\in\Dom\calH_u$. 
By Lemma~\ref{lma.e6}(iv), we find that $\Sigma_u f\in\Dom \calH_u$ and $\calH_u \Sigma_u f=\Sigma_u^*\calH_u f$. 
The proof of Theorem~\ref{thm.e1} is complete. \qed

\section{Completing the proof of Theorem~\ref{thm.d1}}\label{sec.g}

\subsection{The spectral properties of $\Sigma_u$}

We use the following Lemma from \cite{GP}:
\begin{lemma}\cite[Lemma~4.3]{GP}
\label{lma.dd1}
Let $\Sigma_0$ be a contraction on a Hilbert space $X$ such that $\Sigma_0$ is unitarily equivalent to the shift operator $S$ on $H^2(\bbT)$. Let $\Sigma$ be another contraction on $X$ with defect indices $(0,1)$. Assume that $\Sigma-\Sigma_0$ is trace class; then the unitary part of $\Sigma$ in the Wold decomposition \eqref{e8} has no absolutely continuous component. 
\end{lemma}

Let us check that $\Sigma=\Sigma_u$ satisfies the hypothesis of this Lemma for any $u\in H^2(\bbR)$. We first note that $L_u$ is trace class and therefore, by the definitions \eqref{e8aa} of $\Sigma_0$ and $\Sigma_u$ and the resolvent identity, the operator
$$
\Sigma_u-\Sigma_0
$$
is also trace class. Secondly, $\Sigma_0$ is unitarily equivalent to the shift operator $S$ by Lemma~\ref{lma.d3}. We conclude that Lemma~\ref{lma.dd1} applies and so the unitary part of $\Sigma_u$ has no a.c. component.

\subsection{The Wold decomposition of $\Sigma_u$ reduces $\calH_u$}
We use the following Lemma from \cite{GP}:
\begin{lemma}\cite[Lemma~4.4]{GP}
\label{lma.dd2}
Let $\Sigma$ be a contraction on a Hilbert space $X$ satisfying the hypotheses of Lemma~\ref{lma.dd1}. Let $\calH$ be an anti-linear operator on $X$ with a dense domain $\Dom \calH$ with the following properties: 
\begin{enumerate}[\rm (a)]
\item
symmetry
$$
\jap{\calH f,g}=\jap{\calH g,f}, \quad \forall f,g\in\Dom \calH;
$$
\item
if for some $f,h\in X$ we have 
$$
\jap{\calH g,f}=\jap{h,g}, \quad \forall g\in\Dom\calH,
$$
then $f\in\Dom\calH$ and $\calH f=h$;
\item
$\Sigma\Dom\calH\subset\Dom\calH$ and 
$$
\Sigma^*\calH=\calH\Sigma
$$
on the domain of $\calH$. 
\end{enumerate}
Then the Wold decomposition \eqref{e8} of $\Sigma$ reduces $\calH$, i.e. if 
$P^{\rm (cnu)}$  is the orthogonal projection in $X$ onto the subspace $X^{\rm (cnu)}$ in \eqref{e8}, then $P^{\rm (cnu)}\Dom\calH\subset\Dom\calH$ and $\calH P^{\rm (cnu)}(\Dom\calH)\subset X^{\rm (cnu)}$. 
\end{lemma}
Let us check that $\calH=\calH_u$ and $\Sigma=\Sigma_u$ satisfy the hypotheses of this lemma. 
Conditions (a) and (b) hold by Lemma~\ref{lma.e6}(iii),(iv).
Condition (c) is exactly the commutation relation of Theorem~\ref{thm.e1}.

\subsection{Proof of Theorem~\ref{thm.d1}}
As discussed in Section~\ref{sec.e4}, it suffices to prove that  the vector $p=e^{-itH_u^2}u$ belongs to the c.n.u. subspace in the Wold decomposition \eqref{e8} of the contraction $\Sigma_u$. Let us prove this.

\emph{Step 1:} let us check the identity
$$
\lim_{\eps\to0_+} (I-\Sigma_u)\chi_\eps=\sqrt{4\pi} q.
$$
Using the resolvent identity and the definition of the reproducing kernel, we find
\begin{align*}
\lim_{\eps\to0_+} (I-\Sigma_u)\chi_\eps
&=
2i\lim_{\eps\to0_+}(A+L_u+i)^{-1}\chi_\eps
\\
&=
2i\lim_{\eps\to0_+}(I-(A+L_u+i)^{-1}L_u)(A+i)^{-1}\chi_\eps
\\
&=4\pi (I-(A+L_u+i)^{-1}L_u)v_i=\sqrt{4\pi}q.
\end{align*}

\emph{Step 2:} let us check the inclusion
$$
(I-\Sigma_u^*)p\in X^{\rm (cnu)}.
$$
First, we have $u=\lim_{\eps\to0_+}H_u\chi_\eps$ and therefore 
$$
p=e^{-itH_u^2}u=\lim_{\eps\to0_+}e^{-itH_u^2}H_u\chi_\eps=\lim_{\eps\to0_+} \calH_u\chi_\eps.
$$
Next, by the commutation relation \eqref{e1b} we have the identity
$$
(I-\Sigma_u^*)\calH_u\chi_\eps=\calH_u(I-\Sigma_u)\chi_\eps
$$
for each $\eps>0$. The l.h.s. here converges to $(I-\Sigma_u^*)p$ as $\eps\to0_+$. Consider the r.h.s. here. By the previous step, we have
$$
(I-\Sigma_u)\chi_\eps\to\sqrt{4\pi}q\, . 
$$
Since $\calH_u$ is closed (see Lemma~\ref{lma.e6}(i)), it follows that $q\in\Dom \calH_u$ and 
$$
\sqrt{4\pi}\calH_u q=\lim_{\eps\to0_+}(I-\Sigma_u^*)\calH_u\chi_\eps
=(I-\Sigma_u^*)p.
$$
Since $q$ is in the defect subspace of $\Sigma_u$ (see Lemma~\ref{lma.d4}), we have $q\in X^{\rm (cnu)}$. By using Lemma~\ref{lma.dd2} (this is a crucial point in the proof!) we conclude that $\calH_u q\in X^{\rm (cnu)}$ and therefore $(I-\Sigma_u^*)p\in X^{\rm (cnu)}$. 

\emph{Step 3:} Since $\Sigma_u^*$ is a Cayley transform of a maximal dissipative operator, we have $\Ker (I-\Sigma_u^*)=\{0\}$. From here and the previous step of the proof, it follows that $p\in X^{\rm (cnu)}$. 
The proof of Theorem~\ref{thm.d1} is complete. 
\qed

\appendix

\section{Operator theoretic background}

We recall key operator theoretic statements that we use throughout the text. For details, see \cite{HillePh} and \cite{Davies} for semigroup theory and  \cite{NF} for the theory of contractions and dissipative operators. 

\subsection{Dissipative operators} 
Let $X$ be a Hilbert space with the inner product $\jap{\cdot,\cdot}$. An operator $\calA$ with a domain $\Dom \calA\subset X$ is called \emph{dissipative}, if $\Im \jap{\calA f,f}\geq0$ for all $f\in\Dom\calA$. A dissipative operator $\calA$ is called \emph{maximal}, if it has no dissipative extensions. 
Every dissipative operator has a maximal dissipative extension. A maximal dissipative operator is closed. 
(The reader may wish to keep in mind that the distinction between dissipative and maximal dissipative operators is similar to the distinction between symmetric and self-adjoint operators.)

Let $\calA$ be a dissipative operator. 
A simple calculation shows that 
$$
\norm{(\calA+z)f}\geq\Im z\norm{f} 
$$ 
for any $\Im z>0$ and $f\in\Dom \calA$, and in particular the kernel of $\calA+z$ is trivial. Furthermore, $\calA$ is maximal if and only if $\Ran(\calA+z)=X$ for some (and then for all) $\Im z>0$; in this case, by the above estimate,  $\calA+z$ has a bounded inverse and the resolvent norm estimate
\begin{equation}
\norm{(\calA+z)^{-1}}\leq 1/\Im z, \quad \Im z>0
\label{b0a}
\end{equation}
holds true. 

If $\calA$ is maximal dissipative, then so is $-\calA^*$.

Let $\calA$ be a maximal dissipative operator, and let $L$ be a bounded dissipative operator (in particular, $L$ may be self-adjoint). Then the operator $\calA+L$, defined on $\Dom\calA$, is also maximal dissipative (see e.g. \cite[Section 3.1]{Davies}).

\subsection{Semigroups and their generators}

Let $\calS_\tau$, $\tau\geq0$, be a strongly continuous semigroup of contractions (operators of norm $\leq1$) on a Hilbert space $X$. The \emph{infinitesimal generator} of $\calS_\tau$ is the maximal dissipative operator
$$
\calA f=\lim_{\tau_\to0_+}\tfrac1{i\tau}(\calS_\tau f-f)
$$
(where the limit is in the norm of $X$) defined on the dense domain of elements $f$ where the above limit exists. 
Conversely, if $\calA$ is a maximal dissipative operator, then there exists a unique semigroup of contractions $\calS_\tau$ such that $\calA$ is the generator of $\calS_\tau$; notation:  $\calS_\tau=e^{i\tau\calA}$. 
If $\calA$ is the infinitesimal generator of $\calS_\tau$, then $-\calA^*$ is the infinitesimal generator of $\calS_\tau^*$.

\subsection{Cayley transform}
Let $\calA$ be a maximal dissipative operator on $X$. 
The \emph{Cayley transform} of $\calA$ (corresponding to the point $i\in\bbC_+$) is the contraction
$$
\mathcal{T}=(\calA-i)(\calA+i)^{-1}. 
$$
The contraction $\mathcal{T}$ has the property that $1$ is not an eigenvalue of $\mathcal{T}$ and $\Ran(I-\mathcal{T})=\Dom\calA$. 
Finally,  
$$
\mathcal{T}^*=(\calA^*+i)(\calA^*-i)^{-1}
$$
is the Cayley transform of $-\calA^*$.

\subsection{Contractions and the Wold decomposition}
Let $\Sigma$ be a contraction on a Hilbert space, i.e. an operator with the norm $\leq 1$. 
The \emph{defect indices} of $\Sigma$ is the (ordered) pair of integers
$$
\bigl(\rank(I-\Sigma^*\Sigma), \rank(I-\Sigma\Sigma^*)\bigr).
$$
In particular, any unitary operator has the defect indices $(0,0)$, any isometry has the defect indices $(0,k)$ with $k\geq0$ and the shift operator $S$ in $H^2(\bbT)$ has the defect indices $(0,1)$. 

A contraction $\Sigma$ is called \emph{completely non-unitary} (c.n.u.), if $\Sigma$ is not unitary on any of its invariant subspaces. For example, the shift operator $S$ is c.n.u.
The following result is a particular case of the \emph{Wold decomposition} of an isometry (see e.g. \cite[Theorem I.1.1]{NF}).

\begin{theorem}[Wold decomposition]\label{thm.Wold}
Let $\Sigma$ be an isometry on a Hilbert space $X$ with defect indices $(0,1)$, and let $q\in\Ran(I-\Sigma\Sigma^*)$, $q\not=0$. 
Then $X$ can be represented as an orthogonal sum 
$X=X^{\rm (u)}\oplus X^{\rm(cnu)}$, such that 
\begin{equation}
\Sigma=\begin{pmatrix} 
\Sigma^{\rm (u)} & 0
\\
0 & \Sigma^{\rm (cnu)}
\end{pmatrix}
\quad \text{ in $X^{\rm (u)}\oplus X^{\rm (cnu)}$,}
\label{e8}
\end{equation}
where $\Sigma^{\rm (u)}$ is unitary and $\Sigma^{\rm (cnu)}$ is c.n.u. 
Moreover, $X^{\rm (cnu)}$ coincides with the closed linear span of $\{\Sigma^m q\}_{m=0}^\infty$ and $\Sigma^{\rm (cnu)}$ is unitarily equivalent to the shift operator $S$ on $H^2(\bbT)$. 
\end{theorem}

\end{document}